\documentclass[11pt,a4paper,english,reqno]{amsart}
\usepackage{amsmath,amssymb,amsfonts,epsfig,mathrsfs,hyperref}
\usepackage[T1]{fontenc}

\usepackage{color}
\usepackage{array}
\usepackage{amsthm}
\usepackage{amstext}
\usepackage{graphicx}
\usepackage{setspace}
\usepackage{esint}
\usepackage[margin=2.5cm]{geometry}
\usepackage{bbm}
\usepackage{color}
\usepackage{enumitem}
\allowdisplaybreaks[4]

\usepackage{amscd,psfrag}
\usepackage{yhmath}
\usepackage[mathscr]{eucal}

\usepackage{slashed}

\makeatletter
\pdfpageheight\paperheight
\pdfpagewidth\paperwidth

\usepackage{comment}

\usepackage{epstopdf}
\usepackage{chngcntr}
\usepackage{mathrsfs}

\usepackage{indentfirst}	

\usepackage[normalem]{ulem}
\theoremstyle{plain}

\newtheorem{definition}{Definition}[section]
\newtheorem{theorem}[definition]{Theorem}
\newtheorem*{theorem*}{Theorem}

\newtheorem{remark}[definition]{Remark}

\newtheorem*{remark*}{Remark}
\newtheorem*{sideremark*}{Side Remark}\newtheorem*{mt*}{Main Theorem}

\newtheorem*{claim*}{Claim}
\newtheorem*{q*}{Question}
\newtheorem{lemma}[definition]{Lemma}

\newtheorem*{corollary*}{Corollary}
\newtheorem*{proposition*}{Proposition}

\newtheorem{proposition}[definition]{Proposition}

\newcommand{\R}{\mathbb{R}}

\newcommand{\na}{\nabla}
\newcommand{\dd}{{\rm d}}
\newcommand{\p}{\partial}
\newcommand{\e}{\varepsilon}
\newcommand{\emb}{\hookrightarrow}

\newcommand{\G}{\Gamma}
\newcommand{\M}{\mathcal{M}}
\newcommand{\dvg}{\dd V_g}

\newcommand{\two}{{\rm II}}

\newcommand{\loc}{{\rm loc}}

\newcommand{\stwo}{{\mathbf{S}^2}}
\newcommand{\gbar}{{\overline{g}}}

\newcommand{\E}{\mathscr{E}}
\newcommand{\V}{{\mathbf{V}}}
\newcommand{\U}{{\mathbf{U}}}

\newcommand{\J}{{\mathscr{J}}}

\newcommand{\ball}{{\bf B}}
\newcommand{\met}{{\mathfrak{m}}}
\newcommand{\net}{{\mathfrak{n}}}

\newcommand{\XX}{{\bf X}}
\newcommand{\bigo}{\mathcal{O}}
\newcommand{\littleo}{\mathfrak{o}}

\newcommand{\Z}{{\mathcal{Z}_g}} 

\newcommand{\jone}{{J^{\rm I}_\e}}
\newcommand{\jtwo}{{J^{\rm II}_\e}}
\newcommand{\GG}{{\mathcal{G}}}

\allowdisplaybreaks[4]

\def\XXint#1#2#3{{
\setbox0=\hbox{$#1{#2#3}{\int}$}
\vcenter{\hbox{$#2#3$}}
\kern-.6\wd0}}

\newcommand{\mres}{%
\mathbin{%
\vrule height 1.6ex depth 0pt width 0.13ex%
\vrule height 0.13ex depth 0pt width 1.3ex}}

\newcommand{\dvge}{{\dd V_{g_\e}}}
\newcommand{\mc}{{\M_{\rm comp}}}

\newcommand{\B}{{\mathbf{B}}}

\title{On the degenerate Weyl problem on isometric immersions}

\author{Siran Li}

\address{Siran Li: School of Mathematical Sciences $\&$ CMA-Shanghai, Shanghai Jiao Tong University, No.~6 Science Buildings,
800 Dongchuan Road, Minhang District, Shanghai, China (200240)}

\email{\texttt{siran.li@sjtu.edu.cn}}

\author{Xiangxiang Su}
\address{Xiangxiang Su: School of Mathematical Sciences, Shanghai Jiao Tong University, No.~6 Science Buildings,
800 Dongchuan Road, Minhang District, Shanghai, China (200240)}
\email{\texttt{sjtusxx@sjtu.edu.cn}}


\date{\today}

\begin{document}

\maketitle

\begin{abstract}
This paper is concerned with the Weyl problem, \emph{i.e.}, the existence of isometric immersions or embeddings of two-spheres into the three-dimensional Euclidean space or general ambient three-manifolds. We establish two results on the degenerate Weyl problem, namely when the Gaussian curvature $K_g$ is only nonnegative rather than strictly positive. First, for a smooth metric $g$ on the two-sphere, if $K_g$ is strictly positive except at finitely many points where the Hessian of $K_g$ is positive definite, then $g$ admits a global $C^{2,1}$-isometric embedding into $\R^3$. It appears to be the first result on the degenerate Weyl problem with purely intrinsic conditions on $g$. Second, for a general simply-connected ambient three-manifold $(\M,\gbar)$, if $K_g \geq K_0 \geq {\rm sec}_\gbar$ for some constant $K_0$, $(K_g-K_0)^{-1/2} \in L^p$ with $p \geq 2$, and a certain uniform pinching condition holds for approximate nondegenerate isometric immersions, then there exists a $W^{3,p}$-isometric immersion. Alongside we also resolve the nondegenerate Weyl problem (\emph{i.e.}, when $K_g>0$) into general simply-connected ambient three-manifolds for $g$, $\gbar \in C^{2,1}$.
\end{abstract}

\section{Introduction}

\subsection{The Weyl problem}
Isometric immersions or embeddings of surfaces in three-dimensional ambient manifolds have long been a central problem in classical differential geometry and global analysis. Investigations on the existence, uniqueness, stability, and regularity of isometric immersions have played a significant role in the development of geometric analysis and nonlinear PDEs (partial differential equations), and have also provided the theoretical foundations for various problems in physics and applied mathematics, most notably in mathematical relativity and nonlinear elasticity. See Han--Hong~\cite{hh} for a comprehensive exposition.

An outstanding problem in this field, now known as the ``\emph{Weyl problem}'', asks about the existence of isometric immersions of a surface $(\Sigma,g)$ with positive Gaussian curvature $K_g$ diffeomorphic to $\stwo$, the two-sphere into $\R^3$, the Euclidean three-space. This classical problem was first investigated in \cite{w} by H. Weyl (1916), in which \textit{a priori} bounds for the mean curvature in terms of four derivatives of $g$ and the positive lower bound for $K_g$ have been obtained. The Weyl problem serves as the mathematical foundation for the definition of quasi-local masses in cosmology, \textit{e.g.}, the ADM~\cite{ADM59, ADM62}, Brown--York~\cite{by, st}, Scheon--Yau~\cite{syau}, Liu--Yau~\cite{ly1, ly2}, and Wang--Yau~\cite{wy} masses.

In 1938, Lewy~\cite{lewy} solved the Weyl problem for the real-analytic  metric $g$ with strictly positive Gaussian curvature. Nirenberg (1953) extended Lewy's theorem to $g\in C^4$ by establishing the regularity theory for Monge--Amp\`{e}re equations~\cite{N53}. Heinz (1962) further settled in~\cite{h} for $g\in C^3$ (\textit{cf}. also Schulz~\cite{Schulz}; F.-H. Lin~\cite{FHLin}). Aleksandrov and Pogorelov \cite{a,p} obtained generalised solutions to the Weyl problem by passing to limits of convex polyhedra.

Starting from the 1990s, the \emph{degenerate Weyl problem}, \textit{i.e.}, the case that $g$ has nonnegative Gaussian curvature $K_g \geq 0$, has also been studied in the literature. Guan--Li~\cite{GL} and Hong--Zuily~\cite{hz} established the existence of $C^{1,1}$-isometric immersions for $g \in C^4$. The regularity issue of the degenerate Weyl problem has been studied, under hypotheses imposed on the local degeneracy behaviour of the Gaussian curvature, in the framework of degenerate Monge--Amp\`{e}re equations. Notable developments in this direction can be found in P. Daskalopoulos--Savin \cite{DS}, Guan \cite{Guan97}, Guan--Sawyer \cite{GS}, and Jiang \cite{jiang}, among other references. Meanwhile, it has been known that there are  obstructions to the existence of isometric immersions with $C^3$- or higher regularity for the degenerate Weyl problem. Indeed, Burago--Shefel'~\cite{bs} and Iaia~\cite{iaia} constructed  examples for real-analytic metrics on $\stwo$ with positive Gaussian curvature everywhere except at one point but admit \emph{no} global $C^3$-isometric immersions into $\R^3$. 

On the other hand, the existence of isometric immersions of $\stwo$ into a three-dimensional, simply-connected ambient manifold $(\M,\gbar)$ other than $\R^3$ has also been widely studied in the literature; we shall refer to this as the ``\emph{generalised Weyl problem}''. More precisely, it asks about the existence of isometric immersions of  $(\stwo,g)$ with intrinsic curvature $\geq K_0$ into a $3$-dimensional simply-connected Riemannian manifold $(\M,\gbar)$ with sectional curvature $\leq K_0$; here $K_0$ is a finite real number. For $\M=\mathbb{H}^3$, the hyperbolic 3-space, Pogorelov~\cite{p2} proved the existence of isometric immersions for the Gaussian curvature $K_g >-1$, and by Lin--Wang~\cite{lw} for $K_g\geq -1$; see also Chang--Xiao \cite{cx}. We refer the reader to Pogorelov \cite{p}, Guan--Lu \cite{glu}, Lu~\cite{LuWeyl}, and Li--Wang \cite{li-wang} for isometric immersions in more general ambient manifolds.

Labourie (1989) recast the generalised Weyl problem in a fundamentally new framework based on pseudoholomorphic curves~\cite{Labourie}. The existence and regularity issues for the nondegenerate case are consequences of the Schwarz' lemma \textit{\`{a} la} Gromov. Adopting this framework and using elliptic PDE theory, the first named author~\cite{LiWeyl} studied the nondegenerate generalised Weyl problem for $g,\gbar\in C^3$ and obtained tentative results for the degenerate case.

\subsection{Main results on degenerate Weyl problem}
In this paper, we study the regularity issues in the \emph{degenerate} Weyl problem, for isometric immersions of $\stwo$ into both $\R^3$ and general simply-connected ambient 3-manifolds $(\M,\gbar)$. One primary motivation is the following question in Guan--Li~\cite{GL}: 
\begin{quote}
   \emph{Under what conditions on a smooth metric $g$ on $\stwo$ with nonnegative Gaussian curvature is there a global $C^{2,\alpha}$ isometric embedding into $\R^3$ for some $\alpha$, or even a global $C^{2,1}$ isometric embedding?}
\end{quote}
In particular, we seek simple, geometrically natural sufficient conditions for isometric immersions to have third-order derivatives in some weak sense.

We depart from two rather simple ideas:
\begin{enumerate}
    \item 
In view of the counterexamples in Burago--Shefel'~\cite{bs} and Iaia~\cite{iaia}, we look for conditions on $g$ ensuring that $(\stwo,g)$ admits $C^{2,1}$-isometric immersions in $\R^3$. Several sufficient conditions have been obtained in~\cite{DS, GS, jiang}, but they all require certain uniform controls for the \emph{extrinsic geometries} of isometric immersions of metrics with strictly positive Gaussian curvature that approximate $g$. Therefore, we are interested in finding conditions that only relate to the \emph{intrinsic geometry} of $g$ on the set of degeneracy $\{K_g=0\}\cap\stwo$.

\item 
When the Gaussian curvature tends to zero too rapidly as one approaches the degeneracy set $\{K_g=0\}\cap\stwo$, it is natural to expect the failure of higher regularity. We shall prove that if $K_g^{-1}$ is $L^r$-integrable for some $r \geq 1$, and if degeneracy occurs in an isotropic manner (the two principal curvatures are comparable with each other along some approximating sequence of metrics), then $W^{3,{2r}}$-isometric immersions exist.

\end{enumerate}

In (1) above, the counterexamples in~\cite{bs, iaia} suggest that the regularity of isometric immersions in the degenerate Weyl problem cannot be restored simply by increasing the regularity of $g$. We therefore assume $g \in C^\infty$ in (1). On the other hand, if $g \in C^4$ and $K_g \geq 0$, then $\na K_g = 0$ on the degeneracy set $\{K_g=0\}\cap\stwo$, which forces $K_g^{-1} \notin L^r$ for any $r \geq 1$. Hence, we consider only non-$C^4$-metrics in (2).

The first main result of the paper answers in the affirmative the above question by Guan--Li~\cite{GL}, under purely intrinsic assumptions that $K_g$ degenerates at only isolated points, and the Morse index of each degenerate point is zero. 
\begin{theorem}\label{thm: C2,1 isom emb}
Let $\Sigma$ be a closed smooth surface diffeomorphic to $\stwo$ with a $C^\infty$-metric $g$ whose Gaussian curvature $K_g\geq 0$. Assume that the degeneracy set $\{K_g=0\}$ is finite, on which the Hessian of $K_g$ is positive definite. Then $g$ admits a global $C^{2,1}$-isometric embedding into $\R^3$. 
\end{theorem}

The proof of Theorem~\ref{thm: C2,1 isom emb} exploits the results on  degenerate Monge--Amp\`{e}re equations in Guan--Sawyer~\cite{GS} and Daskalopoulos--Savin~\cite{DS}, as well as Caffarelli's theory~\cite{Caffarelli90, Caffarelli91} on viscosity solutions to nondegenerate Monge--Amp\`{e}re equations applied to a cubic blow-up of the degenerate solution in the radial case, whose behaviour is in turn characterised in~\cite{DS}.

In our second main result, we consider a $W^{3,\infty}$-metric $g$ on $\stwo$, and show that integrability conditions on the inverse of the Gaussian curvature near the degeneracy set ensure third-order differentiability of the isometric immersion into general ambient 3-manifolds. 

\begin{theorem}\label{thm:conditional-W3p}
Let $(\Sigma,g)$ be a Riemannian
surface diffeomorphic to $\stwo$ with $g\in W^{3,\infty}$, and let $(\M,\gbar)$ be a complete and simply-connected 3-dimensional Riemannian manifold with $\gbar \in W^{3,\infty}_\loc$. Suppose that the Gaussian curvature of $(\Sigma,g)$ satisfies $K_g \geq K_0$, and the sectional curvature of $(\M,\gbar)$ satisfies $\sec_{\overline g}\leq K_0$, where $K_0 \in \R$ is a constant. Assume that $\{g_\varepsilon\}_{0<\varepsilon<\varepsilon_0}$ is a family of smooth metrics with strictly positive Gaussian curvature and uniformly bounded $W^{3,\infty}$-norms, such that $g_\e \to g$ in $C^{2,\alpha}(\Sigma)$ for every $0<\alpha<1$, and that $\{g_\e\}$ admits $C^3$-isometric immersions $\{f_\varepsilon\}$ with images contained in a compact set $\mc \subset \M$ independent of $\varepsilon$. 

Then, under the assumptions:
\begin{align}\label{eq:main-integral-assumptions}
&\frac{1}{K_g-K_0}\in L^{\frac{p}{2}}(\Sigma,g)\qquad\text{for some $p \geq 2$},\\
& \label{eq:uniform-global-pinching}
\frac{k_\varepsilon}{H_\varepsilon}\geq\tau\qquad\text{for some $\tau >0$ independent of $\e$ a.e. on $\Sigma$},
\end{align} 
there exists a subsequence of $\{f_\e\}$ that converges weakly in the $W^{3,p}$-topology to an isometric immersion $f:(\Sigma,g)\to (\M,\gbar)$.
\end{theorem}

Note that the limiting isometric immersion $f$ in the above theorem is also of regularity $W^{3,p}$. Moreover, by compactness of embedding $W^{3,p}(\Sigma;\M) \emb \emb C^{1,\beta}(\Sigma;\M)$, the convergence $f_\e \to f$ also holds strongly in $C^{1,\beta}$ for any $\beta \in [0,1[$, as long as $p \geq 2$.

Here and throughout the paper, we introduce the following notation:
\begin{itemize}
\item 
The superscript ${}^\#$ and the  subscript ${}_\#$ denote pullback and pushforward, respectively.
    \item
    Let $f_\e: (\Sigma,g_\e) \to (\M,\gbar)$ be an isometric immersion. Then $\kappa_{1,\e}$ and $\kappa_{2,\e}$ are the principal curvatures associated with $f_\e$, namely the smallest and largest eigenvalues of the shape operator $S_\e$ or the second fundamental form $\two_\e$ of $f_\e$. 
    \item 
    $H_\e$ is the mean curvature of $f_\e$; that is,
    \begin{equation}\label{mean curvature, def}
        H_\e := \frac{\kappa_{1,\e} + \kappa_{2,\e}}{2}.
    \end{equation}
    \item 
    Let $K_{g_\e}$ be the Gaussian curvature of $g_\e$, let ${\rm sec}_{\gbar}$ be the sectional curvature of $(\M,\gbar)$, and  let $\mres$ denote the restriction. Then we set
    \begin{equation}\label{small k-eps, def}
        k_\e := \sqrt{K_{g_\e} - {\rm sec}_{\gbar}\mres \left[(f_\e)_\#(T\Sigma) \right]} \equiv \sqrt{\kappa_{1,\e}\kappa_{2,\e}}.
    \end{equation} 
    We say that $k_\e$ is the square root of the ``\emph{relative Gaussian curvature}'' associated with the isometric immersion $f_\e$. At times we simply refer to $k_\e^2$ as the Gaussian curvature.

    \item 
    By a slight abuse of notation, we set
    \begin{equation}\label{small k, def}
    k := \sqrt{K_g-K_0}. 
    \end{equation}
    
\end{itemize}

Let us comment on the two conditions in Theorem~\ref{thm:conditional-W3p}.

\begin{itemize}
    \item[(i)] In view of the above notation, the condition~\eqref{eq:main-integral-assumptions} in Theorem~\ref{thm:conditional-W3p} can be expressed as
\begin{align} \label{k inverse in Lp}
    k^{-1} \in L^p(\Sigma,g)\qquad\text{for some  $p\geq 2$}.
\end{align}
This is stated for $p \geq 2$, but no degeneracy can occur whenever $p \geq 4$. Indeed, suppose that the degeneracy set $\Z:= \{x\in\Sigma:K_g(x)=K_0\}$ is nonempty. Then for $g \in W^{3,\infty}$ the Gaussian curvature is Lipschitz, so $K_g(x)-K_0 \lesssim d_g(x,x_0)$ for $x_0$ near $\Z$. But then
\begin{align*}
    \int_{\ball_r(x_0)} k^{-p}\,\dvg \gtrsim \int_{\ball_r(x_0)} d_g(x,x_0)^{-\frac{p}{2}}\,\dvg(x) \gtrsim \int_0^r s^{1-\frac{p}{2}}\,\dd s = \infty\qquad\text{for } p \geq 4.
\end{align*}

By a similar argument, if $g\in W^{4,\infty}$ (in particular, for the degenerate Weyl problem into $\R^3$ with $C^4$-metrics; \textit{cf.} Guan--Li~\cite{GL} and Hong--Zuily~\cite{hz}), the condition~\eqref{eq:main-integral-assumptions} or \eqref{k inverse in Lp} never holds for any $p \geq 2$. This is because $\na [K_g(x_0)-K_0]=0$ for any $x_0 \in \Z$, so for $x$ close to $x_0$ one has $k^2(x) =K_g(x_0)-K_0 \lesssim d_g(x,x_0)^2$ by Taylor's expansion. Then  
\begin{align*}
    \int_{\ball_r(x_0)} k^{-p}\,\dvg \gtrsim \int_{\ball_r(x_0)} d_g(x,x_0)^{-p}\,\dvg(x) \gtrsim \int_0^r s^{1-p}\,\dd s = \infty\qquad\text{for } p \geq 2.
\end{align*}

\item[(ii)]
The condition~\eqref{eq:uniform-global-pinching} is not purely intrinsic as in Theorem~\ref{thm: C2,1 isom emb}. The definition of the principal curvatures and mean curvature relies on the existence of the isometric immersions in the case of strictly positive (relative) Gaussian curvature. Observe that
\begin{align*}
    \frac{k_\e}{H_\e} = \frac{\sqrt{\kappa_{1,\e}\kappa_{2,\e}}}{\frac{\kappa_{1,\e}+\kappa_{2,\e}}{2}} = \frac{2}{\sqrt{\frac{\kappa_{1,\e}}{\kappa_{2,\e}}+2+\frac{\kappa_{2,\e}}{\kappa_{1,\e}}}},
\end{align*}
so that
\begin{align*}
  \frac{2}{\sqrt{{3+\frac{\kappa_{2,\e}}{\kappa_{1,\e}}}}}\leq  \frac{k_\e}{H_\e} \leq   \frac{2}{\sqrt{{2+\frac{\kappa_{2,\e}}{\kappa_{1,\e}}}}}.
\end{align*}
Hence, the condition~\eqref{eq:main-integral-assumptions} is equivalent to 
\begin{equation}
    \kappa_{2,\e} \leq c\kappa_{1,\e}\qquad\text{for some $c>0$ independent of $\e$}.
\end{equation}
Since $g_\e$ are uniformly bounded in $W^{3,\infty}$, $\gbar \in W^{3,\infty}_\loc$, and isometric immersions $f_\e$ have images confined in a compact subset, $\kappa_{1,\e}$ and $\kappa_{2,\e}$ are uniformly bounded in $L^\infty$. Thus, by \eqref{eq:main-integral-assumptions}, the principal curvatures tend to zero at the same order on the degeneracy set $\Z$, ruling out anisotropic degenerate scenarios for the shape operator.

\end{itemize}

\subsection{Nondegnerate Weyl problem}
In passing, we also solve the generalised Weyl problem in $(\M,\gbar)$ for the nondegenerate case, \emph{i.e.}, when the relative Gaussian curvature is strictly positive, for metrics of $C^{2,1}$-regularity. 

To put this into perspective, Lewy~\cite{lewy} solved the nondegenerate Weyl problem for $g \in C^\omega$ and $\M=\R^3$, Nirenberg~\cite{N53} for $g \in C^4$ and $\M=\R^3$, and Heinz~\cite{h} (\textit{cf}. also Schulz~\cite{Schulz} and Lin~\cite{FHLin}) for $g \in C^3$ and $\M=\R^3$. Lu~\cite{LuWeyl} (see also Guan--Lu~\cite{glu}) established the existence of $C^{2,\mu}$-isometric embeddings for $g \in C^3$ and the ambient manifold $\M$ with scalar curvature bounded from below and finitely many boundary components that are minimal surfaces; the primary example being the Schwarzschild spacetime outside the horizon. Moreover, Lu~\cite{LuWeyl} obtained $C^2$-isometric embeddings for $g \in C^{2,{\rm Dini}}$ into space form $\M$ of constant curvature $\kappa$, provided that the scalar curvature of $g$ is larger than $2\kappa$. Also, for a general ambient manifold $\M$, Heinz~\cite{h} obtained a mean curvature
estimate for $g\in C^3$ and the existence of a global convex function on $\M$; see also Dubrovin~\cite{dub}. Using the pseudo-holomorphic curve formulation \textit{\`{a} la} Labourie~\cite{Labourie}, the first named author settled in~\cite{LiWeyl} the generalised Weyl problem for $g, \gbar \in C^3$.

In this paper, we establish the following result on nondegenerate generalised Weyl problem, along the way of proving Theorem~\ref{thm:conditional-W3p}.
 
\begin{theorem} \label{thm:strict-C21-Weyl}
Let $(\Sigma,g)$ be a diffeomorphic $\mathbb S^2$ with $g \in C^{2,1}$, and let $(\M,\gbar)$ be a complete, simply-connected 3-dimensional Riemannian manifold. Assume that for some constants $K_0 \in\R$ and $\delta_0>0$, one has $K_g\geq K_0+\delta_0$ and $\sec_{\overline g}\leq K_0$. Then there exists an isometric immersion $f: (\Sigma,g) \to (\M,\gbar)$ of regularity $W^{3,p}$ for any $p<\infty$.
\end{theorem}

\subsection{Organisation}

The remaining parts of the paper are organised as follows:
\begin{itemize}
    \item 
In \S\ref{sec: prelims} we collect background materials on isometric immersions, the Weyl problem, and its formulation in the framework of pseudoholomorphic curves. 
\item 
In \S\ref{sec: smooth g} we prove Theorem~\ref{thm: C2,1 isom emb} on the degenerate Weyl problem. More precisely, we establish the existence of global $C^{2,1}$-isometric embedding of nonnegatively curved $\Sigma \cong \stwo$ into $\R^3$ under purely intrinsic assumption that the Morse index of each degenerate point for $K_g$ (there are only finitely many of them) is zero. 

\item
Next, in \S\ref{sec: W3p} we prove Theorem~\ref{thm:conditional-W3p} on the degenerate generalised Weyl problem. We establish that if the relative Gaussian curvature degenerates slowly (\textit{i.e.}, with its inverse in suitable $L^r$-norms) near its degeneracy set, and if the uniform pinching condition $k_\e/H_\e \geq \tau>0$ holds for a family of elliptic-regularised isometric immersions, then a $W^{3,p}$-isometric immersion exists.

\item 
Finally, in \S~\ref{sec: strictly elliptic} we prove, via the method of  pseudoholomorphic curve for isometric immersions \textit{\`{a} la} Labourie~\cite{Labourie}, Theorem~\ref{thm:strict-C21-Weyl} on the strictly elliptic generalised Weyl problem, assuming both domain and target metrics with only $C^{2,1}$-regularity.

\end{itemize}

\section{Preliminaries}\label{sec: prelims}
In this section, we collect several simple geometric facts that are frequently used in the later parts of the paper. Throughout, $\Sigma$ denotes a diffeomorphic $\stwo$ with metric $g \in W^{3,\infty}$ and Gaussian curvature $K_g \in W^{1,\infty}$. The regularity of Riemannian metrics are understood componentwise in some $C^\infty$-atlas that is fixed once and for all.

\subsection{Strictly elliptic approximation to degenerate elliptic isometric immersions} 
We start with a few notations. Denote $k^2:=K_g-K_0$ and $\Z:=\Sigma \cap \{K_g=K_0\}$ as in the Introduction. By $A_\e \sim \bigo(\e^j)$ we mean $|A_\e| \leq C\e^j$ for $C$ independent of $\e$, and by $B_\e \sim \littleo(\e^j)$ we mean $|B_\e| /\e^j \to 0 $ as $\e \to 0^+$. Also, by an abuse of notation, we identify $W^{j,\infty}$ with $C^{j-1,1}$ on Euclidean spaces or Riemannian manifolds (and hence identify pointwise with \textit{a.e.} bounds), without explicitly stating the selection of continuous representatives. 

\begin{proposition}\label{prop:curvature-compatible-approximation}
Let $(\Sigma,g)$ be a diffeomorphic $\stwo$ with $g\in W^{3,\infty}$. Assume that $K_g \geq K_0$ for some constant $K_0 \in \R$. There exists a family of $C^\infty$-metrics $\{g_\varepsilon\}_{0<\varepsilon<\varepsilon_0}$ on $\Sigma$ uniformly bounded in $W^{3,\infty}$, such that for some $0<c_0<C_0$ independent of $\e$, it holds that
\begin{align*}
 c_0\bigl(K_g-K_0+\varepsilon\bigr)\leq K_{g_\varepsilon}-K_0\leq C_0\bigl(K_g-K_0+\varepsilon\bigr) \qquad\text{a.e. on }\Sigma,
\end{align*}
and
\begin{align*}
    g_\varepsilon\longrightarrow g \quad\text{in }C^{2,\alpha}(\Sigma)\qquad \text{for any } \alpha \in ]0,1[.
\end{align*}
\end{proposition}

\begin{proof}[Proof of Proposition~\ref{prop:curvature-compatible-approximation}]
Consider metrics conformal to $g$:
\begin{align*}
g_\varepsilon^0:=e^{2\varepsilon\lambda}g
\end{align*}
with $\lambda\in C^\infty(\Sigma)$ to be determined. The Gaussian curvatures of $g^0_\e$ and $g$ are related by \begin{align*}
K_{g_\varepsilon^0} = e^{-2\varepsilon\lambda} \bigl(K_g-\varepsilon\Delta_g\lambda\bigr).
\end{align*}
Hence, as  $e^{-2\varepsilon\lambda} = 1-2\varepsilon\lambda+\bigo(\varepsilon^2)$ by Taylor's expansion, we have
\begin{align}
K_{g_\varepsilon^0}-K_0 = k^2 + \varepsilon \bigl(-\Delta_g\lambda -2K_0\lambda -2\lambda k^2\bigr) + \bigo(\varepsilon^2).
\label{eq:curvature-gap-expansion}
\end{align}
If we can choose $\lambda$ such that
\begin{align} \label{eq:lambda-choice}
-\Delta_g\lambda-2K_0\lambda>0 \qquad\text{on }\Z,
\end{align}
then $K_{g^0_\e}>K_0$ on $\Z$ for $\e>0$ sufficiently small. The theorem thus follows by taking $g_\e :=$ mollification of $g^0_\e$ at a scale $\ll \e$ and invoking the Arzel\`{a}--Ascoli theorem.

To justify the choice of $\lambda$ that satisfies~\eqref{eq:lambda-choice}, note that if $\Z$ is the empty set, then $\lambda \equiv 0$ would work. If $\Z \neq \emptyset$ and $K_0 \neq 0$, then simply choose $\lambda \equiv -{\rm sgn}(K_0)$. It remains to consider $\Z \neq \emptyset$ and $K_0 = 0$. In this case, $\Z \neq \Sigma$ by the Gauss--Bonnet theorem, and hence there exists $h \in C^\infty(\Sigma)$ such that $\int_\Sigma h\,\dvg=0$ and $h>0$ in some neighbourhood of $\Z$. Then, we may solve $\lambda \in C^\infty(\Sigma)$ from $-\Delta_g \lambda = h$ on $\Sigma$ to satisfy~\eqref{eq:lambda-choice}.     \end{proof}

From the PDE perspectives, $K_g > \sec_\gbar$, the Gauss--Codazzi system corresponding to the isometric immersion $f:(\Sigma,g)\emb(\M,\gbar)$ is strictly elliptic. Thus, we shall adopt the following definition from Labourie~\cite{Labourie}. Recall the (relative) Gaussian curvature of $f$:
\begin{equation}\label{relative Gaussian curvature of f, def}
    K := \big(\text{Gaussian curvature of $g$}\big) - \big(\text{Gaussian curvature of $f_\#(T\Sigma)$}\big).
\end{equation}
\begin{definition}
An isometric immersion $f:(\Sigma,g) \emb (\M,\gbar)$ is said to be {\bf $\e$-elliptic} if the (relative) Gaussian curvature $K$ of $f$ satisfies $K\geq \e>0$ everywhere on $\Sigma$. The immersion $f$ or the immersed surface $f(\Sigma)$ is {\bf elliptic} if $f$ is $\e$-elliptic for some $\e$. It is {\bf degenerate elliptic} if $K\geq 0$ everywhere on $\Sigma$. 
\end{definition}

Let $f_\varepsilon:(\Sigma,g_\varepsilon)\longrightarrow(M,\overline g)$ a family of elliptic isometric immersions. Here and throughout, we reserve the symbols   $\nabla^\e$ and $\overline\nabla$ for the Levi-Civita connections of $g_\varepsilon$ and $\overline g$, respectively. For the Riemann curvature tensor corresponding to $\overline\nabla$, we write \begin{align*}
R^\M(X,Y)Z =\overline\nabla_X\overline\nabla_YZ- \overline\nabla_Y\overline\nabla_XZ-\overline\nabla_{[X,Y]}Z
\end{align*}
and $R^\M(X,Y,Z,W)=\overline g\bigl(R^\M(X,Y)Z,W\bigr)$ for vector fields $X,Y,Z,W \in \G(T\M)$.

By ellipticity of $f_\e$, we may choose unit normal vector fields $\nu_\e$ along $f_\e(\Sigma)$ such that the second fundamental form
$$
\mathrm{II}_\varepsilon(X,Y)=\overline g\bigl(\overline\nabla_{df_\varepsilon(X)}df_\varepsilon(Y),\nu_\varepsilon\bigr)
$$
is positive definite. The shape operator $S_\varepsilon$ is given by
$$
\mathrm{II}_\varepsilon(X,Y)=g_\varepsilon(S_\varepsilon X,Y),
$$
whose eigenvalues are the principal curvatures $\{\kappa_{j,\e}:j=1,2\}$, ordered as $0<\kappa_{1,\varepsilon}\leq\kappa_{2,\varepsilon}$. 

The relative Gaussian curvature of $f_\varepsilon$ is $k_\varepsilon(x)^2:=K_{g_\varepsilon}(x)-\sec_{\overline g}\bigl(df_\varepsilon(T_x\Sigma)\bigr)>0$; see \eqref{small k-eps, def}. The Gauss equation reads $k_\varepsilon^2=\det S_\varepsilon=\kappa_{1,\varepsilon}\kappa_{2,\varepsilon}$. In particular, for $\{f_\e\}$ arising from the strictly elliptic approximation as in Proposition~\ref{prop:curvature-compatible-approximation}, as $\sec_{\overline g}\leq K_0$, we then have 
\begin{align}
k_\varepsilon^2 \geq K_{g_\varepsilon}-K_0 \geq c_0(k^2+\varepsilon).
\label{eq:relative-curvature-lower-bound}
\end{align}

\begin{lemma}\label{lem:inverse-curvature-transfer}
Let $k_\e^2$ and $k^2$ be the relative Gaussian curvatures of $f_\e$ and $f$ as in Proposition~\ref{prop:curvature-compatible-approximation}, respectively. There exists a constant $C$ uniform on $\e$ such that
\begin{align*}
\sup_\varepsilon \left\|k_\varepsilon^{-1}\right\|_{L^p(\Sigma,g_\varepsilon)}
\leq C\left\|k^{-1}\right\|_{L^p(\Sigma,g)}.
\end{align*}
\end{lemma}

\begin{proof}[Proof of Lemma~\ref{lem:inverse-curvature-transfer}]
By \eqref{eq:relative-curvature-lower-bound}, we have the pointwise bound
\begin{align*}
k_\varepsilon^{-1} \leq c_0^{-1/2}(k^2+\varepsilon)^{-1/2} \leq c_0^{-1/2}k^{-1}.
\end{align*}
Meanwhile, the $C^1$ (indeed, $C^{2,\alpha}$) norms of $g_\e-g$ tends to zero, and hence the Riemannian volume forms $\dvg$ and $\dvge$ are uniformly equivalent. This proves the assertion.
\end{proof}

\subsection{Uniform bounds}\label{subsec: unif bd}
As in Theorem~\ref{thm:conditional-W3p}, assume throughout that the images of $\{f_\e\}$ lie within an $\e$-independent compact set $\mc\Subset \M$, where $f_\e$ is a family of elliptic isometric immersions obtained, \emph{e.g.}, from Proposition~\ref{prop:curvature-compatible-approximation}. Then, for some $\Lambda \geq 1$ independent of $\e$, we have that
\begin{align} \label{eq:uniform-coordinate-geometry}
\sup_{\varepsilon}\left(\|g_\varepsilon\|_{W^{3,\infty}}+\|g_\varepsilon^{-1}\|_{W^{3,\infty}}\right)+\|\overline g\|_{W^{3,\infty}}+\|\overline g^{-1}\|_{W^{3,\infty}}\leq\Lambda,
\end{align}
with the norms of $\gbar$ and $\gbar^{-1}$ taken over $\mc$. Then, in view of the construction of isothermal coordinates, we immediately deduce the following lemma, whose proof is safely omitted. 

\begin{lemma}\label{lem:uniform-isothermal-charts}
There exists a radius $r_{\rm iso}>0$ independent of $\e$ such that the following holds. For each $x_0 \in \Sigma$, the geodesic ball $\ball^{g_\e}_{16r_{\rm iso}}(x_0) \subset \Sigma$ lies in an isothermal coordinate chart, in which
\begin{align*}
g_\varepsilon=e^{2u_\varepsilon}(dx^2+dy^2).
\end{align*}
In addition, with norms taken over this chart, 
\begin{align} \label{eq:uniform-isothermal-bounds}
\left\|u_\varepsilon\right\|_{W^{2,\infty}}+\left\|e^{2u_\varepsilon}\right\|_{L^\infty}+\left\|e^{-2u_\varepsilon}\right\|_{L^\infty}\leq C_{\mathrm{iso}} <\infty.
\end{align}
In addition, on any geodesic ball of radius at most $8r_{\mathrm{iso}}$, the Euclidean metric and $g_\e$  comparable in $C^0$-topology, modulo constants independent of $\e$. 
\end{lemma}

We also make a simple observation: for an \textit{a priori} fixed atlas $\mathscr{A}$ on $\M$ with respect to which the $W^{3,\infty}_\loc$-norm of $\gbar$ is taken, by a direct Lebesgue number argument, we may require the images of $f_\e$ over $\ball^{g_\e}_{16r_{\rm iso}}(x_0)$ for any $x_0$ to lie in a single chart of $\mathscr{A}$  by shrinking $r_{\rm iso}$ if necessary.

\begin{lemma} \label{lem:uniform-bounds}
Under the assumption~\eqref{eq:uniform-global-pinching}, namely that $\frac{k_\varepsilon}{H_\varepsilon}\geq\tau>0$, we have 
$0<k_\varepsilon\leq C_\tau$, $k_\varepsilon\leq H_\varepsilon\leq\tau^{-1}k_\varepsilon$, and 
\begin{align} \label{eq:k2-Lipschitz}
\left\|\nabla(k_\varepsilon^2)\right\|_{L^\infty(\Sigma,g_\varepsilon)} + \|\mathrm{II}_\varepsilon\|_{L^\infty(\Sigma,g_\varepsilon)}+\|S_\varepsilon\|_{L^\infty(\Sigma,g_\varepsilon)} \leq C_\tau,
\end{align} 
where $C_\tau$ is a uniform constant that may depend on $\tau$ but not on $\e$. 
\end{lemma}

\begin{proof}[Proof of Lemma~\ref{lem:uniform-bounds}]
Both $K_{g_\varepsilon}$ and the ambient sectional curvature on $\mc$ are controlled in~\eqref{eq:uniform-coordinate-geometry}, so $k_\e \leq C_\tau$ follows. By~\eqref{eq:uniform-global-pinching} and the arithmetic-geometric mean inequality, we have  $k_\varepsilon\leq H_\varepsilon\leq\tau^{-1}k_\varepsilon$. As both principal curvatures are positive, we deduce $
|\mathrm{II}_\varepsilon|_{g_\varepsilon} \leq \kappa_{1,\varepsilon}+\kappa_{2,\varepsilon} = 2H_\varepsilon \leq C_\tau$. The same estimate holds for $S_\varepsilon$.

Finally, fix any $x\in\Sigma$ and local $g_\varepsilon$-orthonormal frame $\{e_1,e_2\}$ such that $\nabla e_i\big|_x=0$. Then
$$
\sec_{\overline g}\bigl(df_\varepsilon(T_x\Sigma)\bigr)=R^\M\bigl(df_\varepsilon(e_1),df_\varepsilon(e_2),df_\varepsilon(e_2),df_\varepsilon(e_1)\bigr).
$$
In view of the Gauss formula, one has
$$
\overline\nabla_{df_\varepsilon(X)}df_\varepsilon(e_i)=\mathrm{II}_\varepsilon(X,e_i)\nu_\varepsilon.
$$
for every $X\in T_x\Sigma$. But $R^\M\in W^{1,\infty}$ on $\mc$, so by differentiation it holds \textit{a.e.} that
$$
\left|X\!\left[\sec_{\overline g}\bigl(df_\varepsilon(T\Sigma)\bigr)\right]\right| \leq C\bigl(1+|\mathrm{II}_\varepsilon|\bigr)|X|_{g_\varepsilon}.
$$
Using the $L^\infty$-bound for $\mathrm{II}_\varepsilon$ just obtained, we deduce the global bound
\begin{align*}
\left|X\!\left[\sec_{\overline g}\bigl(df_\varepsilon(T\Sigma)\bigr)\right]\right| \leq C_\tau|X|_{g_\varepsilon} \qquad\text{for each } X \in \G(T\Sigma).
\end{align*}
Moreover, 
$|\nabla K_{g_\varepsilon}|_{g_\varepsilon} \leq C$ thanks to the uniform $W^{3,\infty}$-bound for $g_\e$. Thus, $$
|\nabla(k_\varepsilon^2)|_{g_\varepsilon} \leq |\nabla K_{g_\varepsilon}|_{g_\varepsilon}
 +\left| \nabla\!\left[ \sec_{\overline g}\bigl(df_\varepsilon(T\Sigma)\bigr) \right]\right|_{g_\varepsilon} \leq C_\tau.
$$
\end{proof}
 
\subsection{Pseudo-holomorphic curves}
In the seminal work~\cite{Labourie}, Labourie (1989) studied the generalised Weyl problem in the strictly elliptic regime from the perspectives of pseudo-holomorphic curves (\textit{a.k.a.} $J$-holomorphic curves). It requires that the ambient manifold $\M$ to be simply-connected, in contrast to \cite{LuWeyl, glu}. Let us streamline in this subsection the pseudo-holomorphic curve formulation for the isometric immersion problem in~\cite{Labourie}. 

\subsubsection{$1$-jet of elliptic isometric immersions}

Let $f_\e:(\Sigma,g_\e) \emb (\M,\gbar)$ be an $\e$-elliptic isometric immersion. Let
\begin{equation}\label{pi-Sigma, bundle}
\begin{matrix}
\E:={\rm Isom}(T\Sigma,T\M)\\
\Big\downarrow \pi_\Sigma\\
\Sigma \times \M
\end{matrix}
\end{equation}
be the bundle of isometries  $T\Sigma \to T\M$. Then the 1-jet of $f_\e$, $$j_1f_\e(x) := (x,f_\e(x), d_x{f_\e}),$$ is a pseudo-holomorphic map from $\Sigma$ to $\E$. See Lemma~\ref{lem: one-jet is J-hol} below for details. 

In general, for $f_\e$ not necessarily isometric, $j_1f_\e$ takes values in  the 1-jet bundle $\J^1(\Sigma,\M)$, whose fibres over $\Sigma \times \M$ are $\J^1(\Sigma,\M)\big|_x :=  \Sigma \times \M \times \left(T_x^*\Sigma \otimes T_{f_\e(x)}\M\right)$.

\subsubsection{The effective component of $T\E$ for elliptic isometric immersions}
Observe that  
\begin{equation*}
    T_{\xi_\e}\E = \V_\e\oplus\U_\e\qquad\text{for each } \xi_\e = df_\e=(f_\e)_\#. 
\end{equation*}
The key component here is
\begin{align}\label{Xi, V}
\V_\e = \Big\{ \GG_\e(u,v):= \big(u,\xi_\e(u),\,k_\e\xi_\e(v)\big):\,u,v\in T\Sigma \Big\},
\end{align}
where $k_\e^2$ is the relative Gaussian curvature; \textit{i.e.}, the difference between the Gaussian curvature of $(\Sigma,g)$ and that of $(f_\e)_\#(T\Sigma)$ as in \eqref{small k-eps, def}. 

Notice the following three geometric properties of $\V_\e$:
\begin{itemize}
    \item[(i)]
$\V_\e$  is equipped with an almost complex structure:
\footnote{More precisely, we should emphasise that $J_\e$ is defined over $\{\xi \in \E: k_\e(\xi)>0\}$; but here $f_\e$ is $\e$-elliptic, so we may define $J_\e$ over $\V_\e$. Meanwhile, the definition of $J_\e$ on $\U_\e$ has no impact for our purpose, as the image of $j_1f_\e$ lies entirely in $\V_\e$ whenever $f_\e$ is an isometric immersion.}
\begin{equation}\label{almost cplx structure J}
  \J_\e := \big[\GG_\e(u,v) \longmapsto \GG_\e(v,-u) \big]. 
\end{equation}

\item[(ii)] 
$\V_\e$ is topologised by the following Hermitian metric: 
\begin{equation}\label{herm metric}
\mu_\e\Big(\GG_\e(u_1,v_1),\GG_\e(u_2,v_2)\Big) := k_\e g_\e(u_1,u_2) + k_\e g_\e(v_1,v_2).
\end{equation}

\item[(iii)]
$\V_\e$ is calibrated (\cite[2.10]{Labourie}): there exists a 1-form $\beta_\e \in \Omega^1(O_\e)$, where $O_\e$ is a neighbourhood of $j_1f_\e(\Sigma)$, such that \begin{equation*}
d\beta_\e\big(x, \J_\e(x)\big)>0.
\end{equation*}
It then follows from \emph{le lemme de Schwarz \`{a} la Gromov} (\cite{g, ms}) that, if $j_1f_\e(\Sigma)$ is precompact, then $j_1f_\e$ is smooth with uniformly bounded derivatives of all orders. 
\end{itemize}

\subsubsection{Two complex structures}
We also make use of two natural complex structures arising from the first and second fundamental forms of the $\e$-elliptic isometric immersion $f_\e$, denoted $\jone$ and $\jtwo$, respectively. 
\begin{itemize}
    \item[(i)]
    We set $\jone$ as the $90^{\circ}$ rotation prescribed by $g_\e$, namely
    \begin{equation}\label{J1}
        \jone X:= (f_\e)^\#\left[\nu_\e \times (f_\e)_\#X\right]\qquad\text{for each } X \in \G (T\Sigma), 
    \end{equation}
where $\times$ is the vector product in the three-manifold $(\M,\gbar)$ and $\nu_\e$ is the unit normal vector field along $(f_\e)_\# (T\Sigma)$. Note that $(\jone)^2 = -{\bf Id}_{T\Sigma}$.

    \item[(ii)]
For elliptic isometric immersion $f_\e$, when $\nu_\e$ is suitably chosen, the second fundamental form $\two_\e$ becomes a Riemannian metric on $\Sigma$.\footnote{This observation has also been explored by Lin~\cite{FHLin} in his harmonic mapping formulation for the isometric immersion problem.}  Thus, $    \two_\e\left(\jtwo u, \jtwo v\right) \equiv \two_\e(u,v)$ for any $u,v \in\G(T\Sigma)$, and $(\jtwo)^2=-{\bf Id}_{T\Sigma}$. 
    
\end{itemize}

\begin{lemma}\label{lem: one-jet is J-hol}
When $f_\e:(\Sigma,g) \emb (\M,\gbar)$ is an $\e$-elliptic isometric immersion, the 1-jet is a pseudo-holomorphic map:
   \begin{align*}
       j_1f_\e: \left(\Sigma, \jtwo\right) \longrightarrow \left( \V_\e, \J_\e \right).
   \end{align*}
\end{lemma}

\begin{proof}[Proof of Lemma~\ref{lem: one-jet is J-hol}]
It holds by \cite[2.5 Proposition]{Labourie} that 
\begin{align}\label{identity, d of 1-jet}
    d(j_1f_\e)(u) = \left(u, df_\e(u), k_\e df_\e \left(\jtwo u\right) \right).
\end{align}
By the notation in~\eqref{Xi, V}, one has that
\begin{align*}
    d(j_1f_\e)(u) = \mathcal{G}_\e\left(u, \jtwo(u)\right).
\end{align*}
This together with the definition of $\J_\e$ in~\eqref{almost cplx structure J}, $(\jtwo)^2=-{\bf Id}$, and~\eqref{identity, d of 1-jet} implies that 
\begin{align*}
    \J_\e \Big( d(j_1f_\e)(u)\Big) &= \mathcal{G}_\e  \left(\jtwo u,-u\right)\\
    &= \mathcal{G}_\e  \left(\jtwo u,\left(\jtwo\right)^2 u\right)\\
    &=  d(j_1f_\e)\left(\jtwo u\right)\qquad\text{for any $u \in \G(T\Sigma)$.}
\end{align*}
Hence the assertion follows.  \end{proof}

\subsubsection{Shape operator}
Let $S_\e$ be the shape operator of $f_\e$, namely that
\begin{equation}\label{shape operator, def}
    \two_\e(u,v) = g_\e\left(S_\e u, v\right)\qquad\text{for any } u,v \in\G(T\Sigma).
\end{equation}
It then holds that
\begin{equation}\label{S, jone, jtwo, 1}
    S_\e = k_\e \jone \circ \jtwo,
\end{equation}
where
\begin{equation*}
    k_\e = \sqrt{\det_{g_\e}\,S_\e}.
\end{equation*}
Equivalently, we may recover $\jtwo$ from the shape operator:
\begin{equation}
    \label{S, jone, jtwo, 2}
    \jtwo = -\frac{1}{k_\e}\jone\circ S_\e.
\end{equation}

\subsubsection{PDE for the mean curvature}\label{subsubsec: J-hol}
We now recall from~\cite[2.12, 2.13, and 3.6]{Labourie} the first-order PDE for the mean curvature $H_\e = \frac{1}{2}(\kappa_{1,\e}+\kappa_{2,\e})$:\footnote{Labourie's convention for the mean curvature is $H_\e = \kappa_{1,\e}+\kappa_{2,\e}$ instead.}
\begin{equation}\label{PDE for H'}
   \frac 12  dH_\e \left(\jtwo u\right) - (H_\e^2-k_\e^2) (\omega_\e u) = -\frac{1}{2}dk_\e \left(\jone u\right) -{\rm tr}\left(\jtwo \overline{R}_u \right),
\end{equation}
where
 \begin{align*}
     \overline{R}_u(v) = R^\M\left((f_\e)_\# u,(f_\e)_\# v\right)\nu_\e + \jone \left(R^\M\left((f_\e)_\# u,\jone (f_\e)_\# v\right) \nu_\e\right).
 \end{align*}
Note that $R^\M\left((f_\e)_\# u,\jone (f_\e)_\# v\right) \nu_\e$ is tangent to $(f_\e)_\# (T\Sigma)$. Meanwhile, $\omega_\e \in \Omega^1(T\Sigma)$ is the connection 1-form associated to the principal directions corresponding to $\kappa_{1,\e}$ and $\kappa_{2,\e}$.

In \cite{Labourie} the PDE~\eqref{PDE for H'} has been recast in the following form:
\begin{equation}\label{PDE for H}
    dH_\e\circ\jtwo = H_\e\beta_\e + 2\left(H_\e^2 - k_\e^2\right) \pi_\Sigma^\#\omega_\e\qquad\text{on } \V_\e,
\end{equation}
where $\V_\e$ is the component
 of $T\E$ as in \eqref{Xi, V}, and $\beta_\e \in \Omega^1(T\Sigma)$ given by
 \begin{align*}
     \left( j_1f_\e \right)^\#\beta_\e(u) = -\frac{1}{H_\e}dk_\e \left(\jone u\right) - \frac{1}{2H_\e} {\rm tr}\left(\jtwo \overline{R}_u \right).
 \end{align*}
Note that away from the umbilics, the connection form $\omega_\e$ satisfies
\begin{equation}\label{bound for omega}
    |\omega|_g \leq \frac{|\na S_\e|_g}{\kappa_{2,\e} - \kappa_{1,\e}},
\end{equation}
while near the umbilics, one has 
\begin{align*}
    2\left(H_\e^2 - k_\e^2\right) \omega_\e(u) = \frac{1}{2}{\rm tr}\left( \na_u S_\e \circ\jtwo\right). 
\end{align*}

Equivalently, for the quantity
\begin{equation}\label{W, def}
    W_\e := \frac{1}{H_\e}
\end{equation}
we have the PDE
\begin{align}\label{PDE for W}
    dW_\e\circ \jtwo
&=
-W_\e\beta_\e
-
2\bigl(1-(k_\e W_\e)^2\bigr)\pi_\Sigma^\#\omega_\e\nonumber\\
&= -W_\e\beta_\e
-
2
\left(
\frac{\kappa_{1,\e}-\kappa_{2,\e}}
     {\kappa_{1,\e}+\kappa_{2,e}}
\right)^2\pi_\Sigma^\#\omega_\e.
\end{align}

\section{Intrinsic sufficient condition for $C^{2,1}$-isometric immersions}\label{sec: smooth g}

Guan--Li~\cite{GL} and Hong--Zuily~\cite{hz} proved that any $C^4$-Riemannian metric $g$ on $\stwo$ with nonnegative Gaussian curvature admits a global $C^{1,1}$-isometric embedding into $\R^3$. On the other hand, Iaia~\cite{iaia} constructed an example of a real-analytic metric on $\stwo$ with positive Gaussian curvature everywhere except at one point, which admits only $C^{2,1}$- but not $C^3$-isometric embedding into $\R^3$. Meanwhile, Pogorelov’s counterexample in~\cite{pog} shows that a $C^{2,1}$-metric with nonnegative Gaussian curvature may not admit an $C^2$-isometric embedding. Note that such obstructions to the existence of smooth isometric embeddings are global, as C.-S. Lin~\cite{lin} has shown that for any smooth two-dimensional Riemannian metric with nonnegative
Gaussian curvature, there exists a smooth local isometric embedding into $\R^3$.

The above discussion leads to the following question posed in~\cite{GL}:

\begin{quote}
   \emph{Under what conditions on a smooth metric $g$ on $\stwo$ with nonnegative Gaussian curvature is there a global $C^{2,\alpha}$ isometric embedding into $\R^3$ for some $\alpha$, or even a global $C^{2,1}$ isometric embedding?}
\end{quote}

In the recent paper~\cite{jiang}, X. Jiang gave an affirmative answer to the above problem, under certain mild hypotheses on the extrinsic geometry of the isometric immersion $\XX$ of $g$. More precisely, it is assumed that near the degenerate point of $K_g$, there exists a local smooth\footnote{Indeed, $C^5$ will be enough; see \cite[Corollary~4.1]{jiang}.} coordinate system in which ${\XX}$ is graphical.

\begin{theorem}[Theorem~1.4 in \cite{jiang}]
Assume that we have a \( C^4 \)-metric \( g \) on a ball \( B(O, r) \subseteq \mathbb{R}^2 \) with Gaussian curvature \( K_g \geq 0 \). Assume that a \( C^{1,1} \)-isometric embedding \( \XX : (B(O, r), g) \to \mathbb{R}^3 \) is of the form
\[
\XX : (x, y) \longmapsto \big(x, y, u(x, y)\big)
\]
in some local coordinates \((x, y)\) such that $u(0,0)=u_x(0,0)=u_y(0,0)=0$. If in addition, \( u \) is weakly convex, then \( \XX \in C^{2,1}(B(O, r')) \) for any \( r' < r \).
\end{theorem}

We address Guan--Li's problem under purely \emph{intrinsic} hypotheses imposed only on the metric $g$ but not on its $C^{1,1}$-isometric immersion $\XX$:
\begin{equation}\label{hess}
\operatorname{Hess}_gK_g(P_j)>0\text{ for each $j$}\qquad \text{where }  \{K_g=0\}=\{P_1,\ldots,P_N\};
\end{equation}
\emph{i.e.}, the Gaussian curvature has zero Morse index at each degenerate point. The condition~\eqref{hess} is intrinsic in light of the definition of Hessian and \emph{Theorema Egregium}. 

Our main result in this direction is Theorem~\ref{thm: C2,1 isom emb}, reproduced here:

\begin{theorem*}
Let $\Sigma$ be a closed smooth surface diffeomorphic to $\stwo$ with a $C^\infty$-metric $g$ whose Gaussian curvature $K_g\geq 0$. Assume that the degeneracy set $\{K_g=0\}$ is finite, on which the Hessian of $K_g$ is positive definite. Then $g$ admits a global $C^{2,1}$-isometric embedding into $\R^3$. 
\end{theorem*}

The proof of the Theorem~\ref{thm: C2,1 isom emb} is based on the following complete classification of the local behaviour near each $P_j$, which holds for $g \in C^{4,\alpha}$ for any $\alpha>0$.

\begin{theorem}[Dichotomy]
\label{thm: dichotomy}
Let $(\Sigma,g)$ be a closed smooth surface diffeomorphic to $\stwo$ with $g\in C^{4,\alpha}$ for some $\alpha \in ]0,1[$ and $K_g \geq 0$. Let $P\in\Sigma$ be an isolated zero of $K_g$ satisfying  
\[
K_g(P)=0,
\qquad
\nabla K_g(P)=0,
\qquad
\operatorname{Hess}_gK_g(P)>0.
\]
In addition, let $\XX:(\Sigma,g)\to \mathbb R^3$ be a locally convex $C^{1,1}$-isometric immersion (whose existence is ensured by \cite{GL}). Then there exists a neighbourhood $U \subset \Sigma$ of $P$ such that one of the
following two alternatives holds:

\begin{enumerate}
\item[\rm (i)]
There exists $c_0>0$ such that the mean curvature of $\XX$ satisfies
\[
H\geq c_0
\qquad\text{a.e. in }U.
\]

\item[\rm (ii)]
Denote
\[
u(Q):=\left\langle \nu(P),\XX(Q)-\XX(P)\right\rangle,
\]
where $\nu$ is the outward unit normal to $\XX$. There exists a uniform constant $C$  such that
\[
\left|\nabla_g^2u(Q)\right|
   \leq C\,d_g(Q,P)
\quad\text{and}\quad
\left\|\nabla_g^2u\right\|_{\rm C^{0,1}(U)}
   \leq C\qquad\text{for each $Q \in U$}.
\]
\end{enumerate}
\end{theorem}

In the case of the first alternative, we shall invoke the following result by Guan and Sawyer:
\begin{theorem}[Theorem~2 in \cite{GS}]\label{thm: Guan--Sawyer}
Suppose \((\Sigma, g)\) is a smooth compact oriented surface with Gaussian curvature \(K_g \geq 0\). If \(K_g\) satisfies the finite type condition ($\clubsuit$): 
\begin{quote}
There exist \(m \in \mathbb{N}\) and \(C > 0\) such that
\[
K_g(x, r, \mathbf{p}) \approx |x|^{2m}
\quad \text{for } (x, r, \mathbf{p}) \text{ in compact subsets of }
\Omega \times \mathbb{R} \times \mathbb{R}^4\quad \tag{$\clubsuit$}
\]
\end{quote}
near every vanishing point and if the \(C^{1,1}\)-isometric embedding \(\XX : \Sigma \to \mathbb{R}^3\) has mean curvature \(H\) bounded below by a positive constant \(c\) almost everywhere, then the embedding \(\XX\) is in fact smooth everywhere.
\end{theorem}

Given an isometric immersion $\XX \in C^{1,1}(\Sigma)$, we denote by $A$ the second fundamental form associated with $\XX$, which is well defined in the \emph{a.e.} sense and is essentially bounded. 
\begin{equation}\label{A, def}
    A_{ij} := \left\langle \na_i \na_j \XX,\nu \right\rangle.
\end{equation}
We choose the sign convention for $\nu$ to ensure that $A$ is nonnegative.

Our main Theorem~\ref{thm: C2,1 isom emb} follows readily from the dichotomy Theorem~\ref{thm: dichotomy}.

\begin{proof}[Proof of Theorem~\ref{thm: C2,1 isom emb}]

Away from the finite degenerate set $\{K_g=0\}$, the classical strictly
elliptic regularity theory for the Weyl problem applies, so that $\XX$ is $C^\infty$ thereon. See, \emph{e.g.}, Nirenberg~\cite{N53}.

At each degenerate point $P_j$ let us apply the dichotomy Theorem~\ref{thm: dichotomy}. 
\begin{itemize}
    \item 
In the latter alternative, the second fundamental form $A$ associated with $\XX$ is Lipschitz. To see this, observe the identity (see Equation~\eqref{eq:A-from-u} in the proof of Theorem~\ref{thm: dichotomy} below): $$A =
\frac{\nabla_g^2u}
 {\sqrt{1-|\nabla_gu|_g^2}}.$$ After shrinking the neighbourhood of the degenerate point if necessary, we have $1-|\nabla_gu|_g^2\geq\frac12$. It then follows from the $C^{2,1}$-bound for $u$ that 
\begin{equation*}
|A(Q)|\leq C\,d_g(Q,P_j),
\qquad
|\nabla_gA(Q)|\leq C
\end{equation*}
for \emph{a.e.} $Q$ near $P_j$.

Once we obtain that $A$ is Lipschitz, recall the Gauss--Weingarten equation:
\begin{equation*}
\partial_{ij}\XX
 =
\Gamma_{ij}^{k}\partial_k\XX+A_{ij}\nu.
\label{eq:Gauss-formula}
\end{equation*}
Since $\XX \in C^{1,1}$, we have $\p_k\XX$ and $\nu \in C^{0,1}$. Also, $\G^k_{ij} \in C^\infty$ and $A_{ij} \in C^{0,1}$. Thus, $\p_{ij}\XX \in C^{0,1}$ and hence $\XX \in C^{2,1}$ locally near $P_j$.

    \item 
     In the former alternative, the mean curvature $H$ of $\XX$ satisfies $H\geq c_j>0$ near $P_j$, while the zero Morse index condition implies that
\begin{align*}
K_g(Q) \approx d_g(Q,P_j)^2  \qquad \text{near $P_j$}.
\end{align*}
It then follows from Guan--Sawyer's Theorem~\ref{thm: Guan--Sawyer} that $\XX \in C^\infty$ near $P_j$.
\end{itemize}

The assertion now follows from a finite patching argument.   \end{proof}


\begin{remark}
    The dichotomy Theorem~\ref{thm: dichotomy} is proved for $C^{4,\alpha}$-metrics. In the case of the second alternative, we clearly have the $C^{2,1}$-regularity of $\XX$. However, in the first alternative with the lower bound $H\geq c>0$, we are currently uncertain whether the arguments in Guan--Sawyer~\cite{GS} would yield $\XX \in C^{2,1}$ for $g \in C^{4,\alpha}$. This is because~\cite{GS} crucially relies on Calabi-type calculations as in Rios--Sawyer--Wheeden~\cite{rsw}, which involves taking third-order derivatives of $K_g$.  
\end{remark}

\begin{proof}[Proof of the dichotomy Theorem~\ref{thm: dichotomy}]
We divide the argument into six steps below.

\medskip
\noindent
\textbf{Step 1.} We first reduce our problem to the analysis of a degenerate Monge--Amp\`{e}re equation investigated in Daskalopoulos--Savin~\cite{DS}. Modulo some Euclidean rigid motion in $\mathbb R^3$, the image of $\XX$ near $\XX(P)$
can be represented as a convex graph over its tangent plane:
\[
\XX(Q)=\bigl(z(Q),v(z(Q))\bigr)
\]
where $v(0)=0$ and $
\na v(0)=0$. Since $\XX\in C^{1,1}$ and the differential of the 
projection on tangent plane is nonsingular at $P$, the maps $Q \mapsto z(Q)$ and $z \mapsto q(z):=Q$ are locally $C^{1,1}$ and, in particular, biLipschitz.

The PDE for the Gaussian curvature of a graph reads
\begin{equation}
\det D^2v(z)
 =
K_g(q(z))\bigl(1+|\na v(z)|^2\bigr)^2,
\label{eq:graph-MA}
\end{equation}
which shall be understood in the Aleksandrov
sense. Denote
\[
L=D q(0).
\]
As $q\in C^{1,1}$ at least locally, we have
\begin{equation}
q(z)=L z+R(z),
\qquad
|R(z)|\leq C|z|^2,
\qquad
|\na R(z)|\leq C|z|
\quad\text{a.e.}
\label{eq:q-expansion}
\end{equation}
On the other hand, since the metric $g\in C^{4,\alpha}$, we have $K_g\in C^{2,\alpha}$. Also, by assumption
\[
K_g(P)=0,
\qquad
\na K_g(P)=0.
\]
An application of the Taylor expansion then yields that
\begin{equation}
K_g(Q)
 =
\frac12\operatorname{Hess}_gK_g(P)\left[\exp^{-1}_P(Q),\exp^{-1}_P(Q)\right]
 +E(Q).
\label{eq:K-Taylor}
\end{equation}
The remainder term $E$ is estimated by
\begin{align*}
    |E(Q)|\leq C|Q|^{2+\alpha},
\qquad
|\na E(Q)|\leq C|Q|^{1+\alpha}.
\end{align*}
 
Define 
\[
Q_0(z)
:=
\frac12\operatorname{Hess}_gK_g(P)[L z,L z].
\]
In view of the zero Morse index assumption on the Hessian of $K_g$, the quadratic form $Q_0$ is positive definite. It follows from \eqref{eq:q-expansion} and \eqref{eq:K-Taylor} that
\begin{equation*}
K_g(q(z))=Q_0(z)+E_0(z)
\label{eq:Khat-expansion}
\end{equation*}
with
\begin{equation*}
|E_0(z)|\leq C|z|^{2+\alpha},
\qquad
|\na E_0(z)|\leq C|z|^{1+\alpha}
\quad\text{a.e.}
\label{eq:F-estimates}
\end{equation*}
As $Q_0(z)\approx |z|^2$, the function
\[
b(z):=
\begin{cases}
E_0(z)/Q_0(z)\qquad\text{ for } z\neq 0,\\
0 \qquad\text{ for } z=0
\end{cases}
\]
belongs to $C^{0,\alpha}$. Indeed, observe that $|b(z)|\leq C|z|^\alpha$ and $|\na b(z)|\leq C|z|^{\alpha-1}$ \emph{a.e.}. By separately considering the cases   $|z-w| \approx \max\{|z|,|w|\}$  and $|z-w| \ll \max\{|z|,|w|\}$, we deduce that $|b(z)-b(w)|\leq C|z-w|^\alpha$.

Now, we \emph{claim} that $v$ satisfies a PDE of the form
\begin{equation}
\det D^2v(z)=|z|^2\widetilde a(z),
\label{eq:DS-form}
\end{equation}
where
\begin{align*}
\widetilde a\in C^{0,\alpha}\quad\text{and} \quad
\widetilde a>0.
\end{align*}
Indeed, as $v\in C^{1,1}$ and $\na v(0)=0$, the function $(1+|\na v|^2)^2$
is locally Lipschitz. Recall $Q_0(z)=
\frac12\operatorname{Hess}_gK_g(P)[L z,L z]$. We  then recast the Monge--Amp\`{e}re equation~\eqref{eq:graph-MA} into
\begin{equation*}
\det D^2v(z)=Q_0(z)a(z)
\end{equation*}
with a strictly positive function $a\in C^{0,\alpha}$. Applying to this PDE a fixed nonsingular linear change of variables and multiplication by a positive constant, we arrive at \eqref{eq:DS-form}.

\smallskip
\noindent
\textbf{Step 2.}  Invoking the classification results in Daskalopoulos--Savin~\cite[Theorems~ 1.1 and 1.2; Remark~1.6]{DS}
for the PDE
\[
\det D^2v=|z|^\gamma \widetilde{a}(z)
\]
with $\widetilde{a}\in C^{0,\alpha}$, $\widetilde{a}>0$, and $\gamma=2$, we conclude that $v\in C^{2,\delta}$ holds for some $\delta>0$. Furthermore, precisely one of the following
two cases occurs:
\begin{enumerate}
    \item 
Either the solution $v$ exhibits a balanced cubic behaviour:
\begin{equation}
c|z|^3\leq v(z)\leq C|z|^3;
\label{eq:cubic-growth}
\end{equation}
\item 
or  the solution $v$ exhibits a nonradial behaviour: \begin{equation}
v(z)
 =
\frac{a_0}{12}z_1^4
 +
\frac{1}{2a_0}z_2^2
 +
\bigo\left(
       \bigl(z_1^4+z_2^2\bigr)^{1+\delta}
  \right)
\label{eq:parabolic-expansion}
\end{equation}
in some coordinate system for some $a_0>0$ and $\delta>0$.
\end{enumerate}

In the nonradial case,
the asymptotic expansion~\eqref{eq:parabolic-expansion} gives us 
\[
D^2v(0)
 =
\begin{pmatrix}
0&0\\
0&a_0^{-1}
\end{pmatrix}.
\]
By continuity of $D^2v$ (since $v \in C^{2,\delta}$), it holds that 
\[
\Delta v\geq c>0
\]
in a neighbourhood of the origin.

For the graph $\XX(Q)=\bigl(z(Q),v(z(Q))\bigr)$, we have the well-known identity
\[
H
 =
\frac{
 (1+v_2^2)v_{11}
 -2v_1v_2v_{12}
 +(1+v_1^2)v_{22}
}{
 2(1+|\na v|^2)^{3/2}
}
\]
in the $z$-coordinates. Here and hereafter, we adopt the convention that $H$ is one half of the trace
of the shape operator.  After undoing the fixed nonsingular linear normalisation used to
reduce $Q_0$ to $|z|^2$, we conclude that, in the original
orthonormal tangent plane coordinates, $D^2v(0)$ is a nonzero
positive semidefinite matrix of rank one. Consequently,
\[
\operatorname{tr}D^2v(0)>0.
\]
Since $\nabla v(0)=0$, the graph mean-curvature formula gives
\[
H(0)=\frac12\operatorname{tr}D^2v(0)>0.
\]
By the continuity of $D^2v$, after shrinking the neighbourhood,
\[
H\ge c_0>0.
\]
This proves the alternative \textup{(i)}.

It remains to analyse the balanced cubic behaviour in \eqref{eq:cubic-growth}. 

\smallskip
\noindent
\textbf{Step 3.} Let $\nu$ be the outward unit normal vector field along $\XX$. By assumption, $\nu$ is Lipschitz. Denote the vector
\begin{align*}
    {\bf e}:= \nu(P),
\end{align*}
and define for $Q \in \Sigma$ the scalar function 
\begin{align*}
    u(Q)=\left\langle {\bf e},\XX(Q)-\XX(P)\right\rangle.
\end{align*}
The standard height identities yield that
\begin{equation}
\nabla_g^2u=\langle {\bf e},\nu\rangle A,
\label{eq:height-Hessian}
\end{equation}
where $A$ is the second fundamental form associated with $\XX$, as well as
\begin{equation*}
\langle {\bf e},\nu\rangle^2
 =
1-\left|\nabla_gu\right|_g^2.
\end{equation*}
We thus have 
\begin{equation}
A
 =
\frac{\nabla_g^2u}
 {\sqrt{1-|\nabla_gu|_g^2}}.
\label{eq:A-from-u}
\end{equation}
Now, taking determinants on both sides of \eqref{eq:height-Hessian} and using the Gauss equation~$\det_gA=K_g$, we obtain that
\begin{equation}
\det{}_g\nabla_g^2u
 =
K_g\left(1-|\nabla_gu|_g^2\right).
\label{eq:intrinsic-Darboux}
\end{equation}
In local coordinates this reads
\begin{equation}
\det\left(
 u_{ij}-\Gamma_{ij}^{k}u_k
\right)
 =
K_g\det(g)
\left(
 1-g^{ij}u_i u_j
\right),
\label{eq:coordinate-Darboux}
\end{equation}
where $\{g^{ij}\}=g^{-1}$. This is the Darboux equation.

To proceed, we choose the geodesic normal coordinate system $x=(x^1,x^2)$ centred at $P$. Recall that \[
K_g(P)=0,
\qquad
\nabla_gK_g(P)=0,
\]
whence in dimension two the full Riemann curvature tensor is determined by $K_g$. We thus have the asymptotic expansions
\begin{equation}
g_{ij}(x)=\delta_{ij}+\bigo(|x|^4),
\qquad
\Gamma_{ij}^{k}(x)=\bigo(|x|^3).
\label{eq:flatness}
\end{equation}
More precisely, these expansions mean that
\[
\left|D^m(g_{ij}-\delta_{ij})(x)\right|
\le C|x|^{4-m},
\qquad 0\le m\le4,
\]
and
\[
\left|D^m\Gamma_{ij}^{k}(x)\right|
\le C|x|^{3-m},
\qquad 0\le m\le3,
\]
with the corresponding H\"older estimates at the top order.

Recall that graphing function $v$ is exactly the function $u$ expressed
in the tangent plane coordinates $z$; also, both $z$ and the normal coordinates $x$
are biLipschitz and have nonsingular linear part at the origin. Thus, we may recast the balanced cubic branch~\eqref{eq:cubic-growth} in Step~2 above as
\begin{equation}
c|x|^3\leq u(x)\leq C|x|^3,
\label{eq:u-cubic-growth}
\end{equation}
after shrinking the coordinate neighbourhood if necessary.

\smallskip
\noindent
\textbf{Step 4.} In this step, we  first obtain uniform ellipticity
from the two-dimensional interior regularity theory for the
Monge--Amp\`ere equation in the convex tangent plane coordinates 
introduced in Step~1, and then apply the Evans--Krylov theory to the intrinsic
rescaled Darboux equation.

Recall that, in the original orthonormal tangent plane coordinates
$z$, the graphing function $v$ satisfies
\begin{equation}
\det D^2v(z)=Q_0(z)a(z),
\label{eq:graph-MA-Q0}
\end{equation}
where $Q_0$ is a fixed positive-definite quadratic form and
\begin{equation*}
0<a_*\le a(z)\le a^*,
\qquad
\|a\|_{C^{0,\alpha}(B_{\rho_0})}\le C.
\end{equation*}
In the balanced branch (recall \eqref{eq:cubic-growth}),
\begin{equation*}
c|z|^3\le v(z)\le C|z|^3.
\label{eq:v-cubic-original}
\end{equation*}

For $r>0$, define the cubic blowup:
\begin{equation}
v_r(\zeta):=\frac{v(r\zeta)}{r^3}.
\label{eq:vr-definition}
\end{equation}
Since $Q_0$ is degree-2-homogeneous, Equation~\eqref{eq:graph-MA-Q0} becomes
\begin{equation}
\det D^2v_r(\zeta)
=
Q_0(\zeta)a(r\zeta)
=:f_r(\zeta).
\label{eq:vr-MA}
\end{equation}

Denote momentarily
\[
\mathcal Z_1\Subset\mathcal Z_0
\Subset \mathbb R^2\setminus\{0\}
\]
as two fixed bounded annular domains. For sufficiently small
$r$, \eqref{eq:vr-MA} is defined on a neighbourhood of
$\overline{\mathcal Z_0}$. Since $Q_0$ is positive definite and
$\mathcal Z_0$ stays a positive distance away from the origin, we deduce from the bounds for $a(z)$ that
\begin{equation*}
0<\lambda_0\le f_r\le\Lambda_0
\qquad\text{on }\mathcal Z_0
\end{equation*}
and
\begin{equation}
\|f_r\|_{C^{0,\alpha}(\mathcal Z_0)}
\le C.
\label{eq:fr-Holder}
\end{equation}
This follows from the scaling $[a(r\,\cdot)]_{C^{0,\alpha}(\mathcal Z_0)}
=
r^\alpha
[a]_{C^{0,\alpha}(r\mathcal Z_0)}
\le C$.
Moreover, in view of the definition of $v_r$ in~\eqref{eq:vr-definition} and the cubic bounds for $v$, we have
\begin{equation*}
\|v_r\|_{L^\infty(\mathcal Z_0)}\le C.
\label{eq:vr-Linfty}
\end{equation*}
Since $v_r$ is convex, the standard interior gradient estimate for
convex functions leads to
\begin{equation*}
\|Dv_r\|_{L^\infty(\mathcal Z_1)}
\le C.
\label{eq:vr-gradient}
\end{equation*}

For our purpose, let us recall the following standard two-dimensional interior
Monge--Amp\`ere estimate: if $V$ is a convex Aleksandrov solution to 
\begin{align*}
    \det D^2V=f\qquad \text{ on } \ball_{2\rho},
\end{align*}
such that
\[
0<\lambda\le f\le\Lambda,
\qquad
\|f\|_{C^{0,\alpha}(\ball_{2\rho})}\le M,
\qquad
\operatorname{osc}_{\ball_{2\rho}}V\le M_0,
\]
then
\begin{equation*}
\|V\|_{C^{2,\alpha}(B_\rho)}
\le
C(\rho,\lambda,\Lambda,M,M_0,\alpha).
\label{eq:2D-interior-MA}
\end{equation*}
This follows from the  strict convexity and interior
regularity theory of Caffarelli~\cite{Caffarelli90, Caffarelli91}.

Covering $\overline{\mathcal Z_1}$ by finitely many balls compactly
contained in $\mathcal Z_0$, and applying the theory in the previous paragraph to the Monge--Amp\`{e}re Equation~\eqref{eq:vr-MA}, we obtain that
\begin{equation}
\|v_r\|_{C^{2,\alpha}(\mathcal Z_1)}
\le C,
\label{eq:vr-C2a}
\end{equation}
where $C$ is independent of $r$.

Since $v_r$ is convex, \eqref{eq:vr-C2a} gives
\[
0\le D^2v_r\le CI.
\]
Combining this with
\[
\det D^2v_r=f_r\ge\lambda_0
\]
yields the uniform lower bound
\begin{equation}
cI\le D^2v_r\le CI
\qquad\text{on }\mathcal Z_1.
\label{eq:vr-uniform-convexity}
\end{equation}

We next express \eqref{eq:vr-uniform-convexity} in the intrinsic 
normal coordinates. Let
\[
z=\Phi(x)
\]
be the tangent plane projection map expressed in the intrinsic
normal coordinates $x$ centred at $P$. Since the original
immersion is $C^{1,1}$ and its tangent plane projection is
nonsingular at $P$, we have that
\[
\Phi\in C^{1,1},
\qquad
\Phi(0)=0,
\qquad
D\Phi(0)\ \text{is nonsingular}.
\]
In particular,
\begin{equation}
\Phi(x)=D\Phi(0) x+\bigo(|x|^2),
\qquad
D\Phi(x)=D\Phi(0)+\bigo(|x|).
\label{eq:Phi-expansion}
\end{equation}

Define the rescaled coordinate transformation
\begin{equation*}
\Phi_r(y):=\frac{\Phi(ry)}{r}.
\end{equation*}
Then we have
\begin{align*}
&\Phi_r(y)=D\Phi(0)y+\bigo(r|y|^2),\\
&D \Phi_r(y)=D \Phi(ry)=D\Phi(0)+\bigo(r|y|),\\
&D^2\Phi_r(y)=rD^2\Phi(ry)\qquad \text{a.e.}.\label{eq:Phi-r-C2}
\end{align*}
Hence, on every fixed bounded annulus, it holds that
\begin{equation}\label{eq:Phi-r-bounds} 
\left\|D\Phi_r\right\|_{L^\infty}
+
\left\|(D\Phi_r)^{-1}\right\|_{L^\infty}
\le C,
\qquad
\left\|D^2\Phi_r\right\|_{L^\infty}
\le Cr.
\end{equation}

Let us fix three annuli:
\begin{align*}
    \mathcal A_0
 =
\left\{y:\frac14<|y|<4\right\},
\qquad
\mathcal A_1
 =
\left\{y:\frac12<|y|<2\right\},\qquad \mathcal A_2
 =
\left\{y:\frac34<|y|<\frac32\right\}.
\end{align*}
By enlarging $\mathcal Z_0$ and choosing $\mathcal Z_1$ appropriately,
one may arrange that $\Phi_r(\mathcal A_0)\Subset\mathcal Z_1$ for every sufficiently small $r>0$.

The intrinsic height function $u$ is related to $v$ by
\[
u(x)=v(\Phi(x)).
\]
Hence, 
\begin{align*}
    & u_r(y)=v_r(\Phi_r(y)),\\
    &D u_r
=
(D\Phi_r)^\top
D v_r(\Phi_r),\\
&D^2u_r
=
(D\Phi_r)^T
D^2v_r(\Phi_r)
D\Phi_r + \sum_{\ell=1}^2
(v_r)_\ell(\Phi_r)
D^2(\Phi_r)^\ell\qquad a.e..
\end{align*}

The bounds~\eqref{eq:vr-uniform-convexity}
and \eqref{eq:Phi-r-bounds} imply that
\begin{equation}
\|D u_r\|_{L^\infty(\mathcal A_0)}
\le C.
\label{eq:ur-gradient-bound}
\end{equation}
Also observe that
\[
cI
\le
(D\Phi_r)^\top
D^2v_r(\Phi_r)
D\Phi_r
\le
CI,
\]
while
\begin{align*}
    \left\|\sum_{\ell=1}^2
(v_r)_\ell(\Phi_r)
D^2(\Phi_r)^\ell \right\|_{L^\infty(\mathcal{A}_0)} \leq Cr.
\end{align*}
Hence, after reducing
$r$ if necessary,
\begin{equation}
cI\le D^2u_r\le CI
\qquad\text{a.e. on }\mathcal A_0,
\label{eq:ur-uniform-Hessian}
\end{equation}
with constants independent of $r$.

We now return to the Darboux equation~\eqref{eq:coordinate-Darboux}. A change of variables yields the following PDE for $u_r$, referred to as the rescaled Darboux equation in the sequel:
\begin{align}
\det M_r
 &=
K_r(y)\det g(ry) 
\left\{
 1-r^4g^{ij}(ry)(u_r)_i(u_r)_j
\right\},
\label{eq:rescaled-Darboux}
\end{align}
where
\begin{equation}
    M_r:=D^2u_r-r\Gamma(ry)Du_r\label{def: Mr}
\end{equation}
and 
\begin{equation*}
K_r(y)=r^{-2}K_g(ry).
\end{equation*}
We abbreviate $\G D u$ as the $2\times 2$ matrix $\left\{\G^k_{ij}\p_k u\right\}_{i,j}$, and denote $\|\G\| \equiv \max_{i,j,k \in \{1,2\}}\left\|\G^k_{ij}\right\|$ for any norm $\|\cdot\|$. Noticing that $\Gamma(x)=\bigo(|x|^3)$ in the vicinity of a degenerate point, so  by \eqref{eq:ur-gradient-bound} one has the bound
\begin{equation}
\left\|
r\Gamma(ry)D u_r
\right\|_{L^\infty(\mathcal A_0)}
\le Cr^4.
\label{eq:connection-small}
\end{equation}
Combining \eqref{eq:ur-uniform-Hessian} and
\eqref{eq:connection-small}, we obtain the inequalities for quadratic forms:
\begin{equation}
c_1I\le M_r\le C_1I
\qquad\text{on }\mathcal A_0,
\label{eq:Mr-uniform-ellipticity}
\end{equation}
for constants $c_1,C_1>0$ independent of $r$. Moreover, using \eqref{eq:ur-gradient-bound} again, we obtain that
\begin{equation}
1-r^4g^{ij}(ry)(u_r)_i(u_r)_j\ge 1-Cr^4\ge\frac12
\label{eq:qr-lower}
\end{equation}
for all sufficiently small $r>0$.

We are now in a position to apply the Evans--Krylov theorem (\emph{cf.} \emph{e.g.}, \cite{cc}). The rescaled Darboux Equation~\eqref{eq:rescaled-Darboux} is tantamount to
\begin{equation}
\mathscr{F}_r(y,D u_r,D^2u_r)=0,
\label{eq:Fr-equation}
\end{equation}
where
\begin{align*}
\mathscr{F}_r(y,\xi,\net)
&:=
\log\det\left(
\net-r\Gamma(ry)\xi
\right) -\log\!\left(
K_r(y)\det g(ry)
\right)
\nonumber\\
&\quad
-\log\!\left(
1-r^4g^{ij}(ry)\xi_i \xi_j
\right).
\end{align*}
The operator $\mathscr{F}_r$ is concave in the $\net$ variable. Its linearisation along a solution is
\begin{equation}
\frac{\partial \mathscr{F}_r}{\partial \net_{ij}}
=
(M_r^{-1})^{ij},
\label{eq:Fr-linearization}
\end{equation}
where we recall $M_r$ from \eqref{def: Mr}. Then, in view of \eqref{eq:Mr-uniform-ellipticity}, the nonlinear differential operator $\mathscr{F}_r$ is uniformly elliptic, with ellipticity constants independent
of $r$. In addition, thanks to the assumptions $g\in C^{4,\alpha}$ and ${\rm Hess}_gK_g(P)>0$, we have
\[
0<c\le K_r\le C,
\qquad
\|K_r\|_{C^{2,\alpha}(\mathcal A_0)}\le C.
\]
From here and \eqref{eq:ur-gradient-bound},
\eqref{eq:ur-uniform-Hessian}, as well as the estimates for the rescaled metric coefficients, we deduce that the $C^{0,\alpha}$-norms of
the coefficients of $\mathscr{F}_r$ are bounded uniformly in $r>0$.

By now, we have only proved the uniform ellipticity along the solution to Equation~\eqref{eq:Fr-equation}. In order to apply the Evans--Krylov theorem, one still needs to extend the
logarithmic Monge--Amp\`ere operator to a globally defined, concave, uniformly elliptic operator. To this end, define the increasing concave
$C^{1}$-function $\psi\colon \mathbb{R}\to\mathbb{R}$ as follows:
\begin{equation*}
 \psi(t):=   \begin{cases}
\displaystyle
\log\frac{c_{1}}{2}
+\frac{2}{c_{1}}\left(t-\frac{c_{1}}{2}\right)\qquad&\text{if }
 t\leq \dfrac{c_{1}}{2},\\
\log t\qquad&\text{if } \dfrac{c_{1}}{2}\leq t\leq 2C_{1},
\\
\displaystyle
\log(2C_{1})
+\frac{1}{2C_{1}}(t-2C_{1})
\qquad&\text{if } t\geq 2C_{1},
    \end{cases}
\end{equation*} 
with $c_1, C_1$ as in \eqref{eq:Mr-uniform-ellipticity}. It follows that 
\[
\frac{1}{2C_{1}}
\leq \psi'(t)\leq \frac{2}{c_{1}}
\qquad\text{for every }t\in\mathbb{R}.
\]
Then, for any symmetric $2 \times 2$ matrix $\met \in {\rm Sym_{2\times 2}(\R)}$, let us define $\psi(\met)$ by the usual spectral
functional calculus and put
\[
\Psi(\met):=\operatorname{tr}\psi(\met)
        =\psi\bigl(\lambda_{1}(\met)\bigr)
         +\psi\bigl(\lambda_{2}(\met)\bigr),
\]
where $\lambda_{1}(\met)$, $\lambda_{2}(\met)$ are the eigenvalues of
$\met$. The operator $\Psi$ is concave on ${\rm Sym_{2\times 2}(\R)}$.
Moreover, for any $\met,\net\in{\rm Sym_{2\times 2}(\R)}$ with $\net \geq 0$, one has that
\[
\frac{1}{2C_{1}}\operatorname{tr}\net
\leq
\Psi(\met+\net)-\Psi(\met)
\leq
\frac{2}{c_{1}}\operatorname{tr}\net.
\]
Hence, $\Psi$ is globally uniformly elliptic. Also $\Psi(\met)=\log\det \met$
whenever $\frac{c_{1}}{2}I\leq \met\leq 2C_{1}I$.

For the fixed solution $u_{r}$, set
\begin{align*}
   &(S_{r})_{ij}(y)
:=
r\Gamma_{ij}^{k}(ry)(u_{r})_{k}(y),\\ 
&B_{r}(y):=K_{r}(y)\det g(ry),\\
&q_{r}(y):=
1-r^{4}g^{ij}(ry)(u_{r})_{i}(u_{r})_{j}.
\end{align*}
Then define
\[
\widehat{\mathscr{F}}_{r}(y,\met)
:=
\Psi\bigl(\met-S_{r}(y)\bigr)
-\log B_{r}(y)-\log q_{r}(y),
\qquad
\met\in\operatorname{Sym}_{2\times 2}(\R).
\]
For every fixed $y$, the operator $\met \mapsto \widehat{\mathscr{F}}_{r}(y,\met)$ is concave and uniformly elliptic, with ellipticity constants
$\frac{1}{2C_{1}}$ and $\frac{2}{c_{1}}$ (both independent of $r$).

By the estimates in~\eqref{eq:ur-gradient-bound} and \eqref{eq:ur-uniform-Hessian}, $u_{r}$ have uniformly
bounded $C^{1,1}$-norms on the annulus $\mathcal{A}_0$. We have proved that
\[
\|S_{r}\|_{C^{0,\alpha}(\mathcal{A}_0)}
+
\|\log B_{r}\|_{C^{0,\alpha}(\mathcal{A}_0)}
+
\|\log q_{r}\|_{C^{0,\alpha}(\mathcal{A}_0)}
\leq C
\]
with $C$ independent of $r$. Thus, for every $\met\in\operatorname{Sym}_{2\times 2}(\R)$ and all $y,y'$ in $\mathcal{A}_0$, it holds that
\[
\left|
\widehat{\mathscr{F}}_{r}(y,\met)-\widehat{\mathscr{F}}_{r}(y',\met)
\right|
\leq C|y-y'|^{\alpha}.
\]
The constant here is
independent of both $\met$ and $r$. On the other hand, by \eqref{eq:ur-uniform-Hessian} we have $c_{1}I \leq D^{2}u_{r}-S_{r} \leq C_{1}I$ \emph{a.e.}, so 
\[
\Psi(M_{r})=\log\det M_{r}
\qquad\text{a.e.}.
\]
Moreover, the rescaled Darboux Equation~\eqref{eq:rescaled-Darboux} yields
\begin{equation}\label{hat-F eq}
\widehat{\mathscr{F}}_{r}\bigl(y,D^{2}u_{r}(y)\bigr)=0
\qquad\text{for a.e. }y \in \mathcal{A}_0.
\end{equation}
As $u_{r}\in C^{1,1}$ and $\widehat{\mathscr{F}}_{r}$ is continuous and uniformly elliptic, $u_{r}$ is also a viscosity solution to \eqref{hat-F eq}.

We may now apply the interior Evans--Krylov estimates~\cite{cc} to the globally uniformly elliptic equation~\eqref{hat-F eq} to conclude that 
\begin{equation}
\|u_r\|_{C^{2,\beta}({\mathcal A}_1)} \le C \left(\|u_r\|_{L^\infty(\mathcal{A}_0)}+1\right)
\le C'
\label{eq:ur-EK}
\end{equation}
for some
$\beta\in(0,\alpha)$, where $C$ and $C'$ are independent of $r$. This establishes both the uniform annular ellipticity and the uniform $C^{2,\beta}$ estimate needed in Step~5 below.\footnote{The important point is that the $C^{1,1}$-change between intrinsic and graph coordinates becomes asymptotically linear after rescaling $\Phi_r(y):=r^{-1}\Phi(ry)$. Therefore, the strong two-dimensional Monge--Ampère estimate available for the genuinely convex graphing function $v_r$ transfers to $u_r$ with an error $\sim\bigo(r)$. This circumvents the potential circularity in the regularity estimates: the application of Evans--Krylov requires uniform ellipticity, while we have to establish ellipticity via similar estimates.}

\smallskip
\noindent
{\bf Step~5.} We now derive the third-order estimates for $u_r$. In this step only, for notational simplicity, we put $u \equiv u_r$. We introduce the following notation:
\begin{align*}
&\Gamma_{r,ij}^{\ell}(y)
:=\Gamma_{ij}^{\ell}(ry),\qquad G_r^{ij}(y)
:=g^{ij}(ry),\\
&J_r(y)
:=\det g(ry),\qquad A_r(y):=K_r(y)J_r(y),\\
&q_r(y,\na u):=1-r^4G_r^{ij}(y)u_i u_j.
\end{align*}
We also introduce the rescaled covariant Hessian:
\begin{equation*}
M_{r,ij}
:=
u_{ij}-r\Gamma_{r,ij}^{\ell}u_\ell.
\end{equation*}
Then we may rewrite Equation~\eqref{eq:rescaled-Darboux} as
\begin{equation}
\det M_r=A_rq_r.
\label{eq:det-Mr}
\end{equation}

We \emph{claim} the following: under the assumptions that ~$A_r>0$, $q_r>0$, and $M_r$ is positive definite, the variables \[
w_k:=\partial_k u,\qquad k \in \{1,2\}
\] satisfy the linear, non-divergence-form PDE:
\begin{equation}
a_r^{ij}(w_k)_{ij}
+b_r^\ell(w_k)_\ell
=f_{r,k},
\label{eq:linear-wk}
\end{equation}
where
\begin{align*}
&a_r^{ij}=(M_r^{-1})^{ij},\\
&b_r^\ell
=
-r\,a_r^{ij}\Gamma_{r,ij}^{\ell}
+
\frac{2r^4}{q_r}G_r^{\ell \alpha}u_\alpha,\\
&f_{r,k}
=
r\,a_r^{ij}
(\Gamma_{r,ij}^{\ell})_{,k}u_\ell
+
(\log A_r)_{,k}
-
\frac{r^4}{q_r}
(G_r^{\alpha\beta})_{,k}u_\alpha u_\beta.
\end{align*} 
All derivatives are taken with
respect to the rescaled variable $y$. Indeed, since $A_r>0$ and $q_r>0$, Equation~\eqref{eq:det-Mr} is equivalent to
\begin{equation*}
\log\det M_r=\log A_r+\log q_r.
\end{equation*}
Differentiation with respect to $y_k$ then leads to
\[
a_r^{ij}\partial_kM_{r,ij}
=
(\log A_r)_{,k}
+
\frac{(q_r)_{,k}}{q_r}.
\] The definition of the rescaled Hessian $M_r$ satisfies
\begin{align*}
\partial_kM_{r,ij}
&=
u_{ijk}
-r\Gamma_{r,ij}^{\ell}u_{\ell k}
-r(\Gamma_{r,ij}^{\ell})_{,k}u_\ell\\
&=
(w_k)_{ij}
-r\Gamma_{r,ij}^{\ell}(w_k)_\ell
-r(\Gamma_{r,ij}^{\ell})_{,k}u_\ell.
\end{align*}
On the other hand,
\begin{align*}
(q_r)_{,k}
&=
-r^4(G_r^{\alpha\beta})_{,k}u_\alpha u_\beta
-r^4G_r^{\alpha\beta}
\left(
u_{\alpha k}u_\beta+u_\alpha u_{\beta k}
\right)\\
&=
-r^4(G_r^{\alpha\beta})_{,k}u_\alpha u_\beta
-2r^4G_r^{\alpha\beta}u_\beta(w_k)_\alpha.
\end{align*}
Direct computation then yields
\begin{align*}
a_r^{ij}(w_k)_{ij}
&+
\left(
-r\,a_r^{ij}\Gamma_{r,ij}^{\ell}
+
\frac{2r^4}{q_r}G_r^{\ell \beta}u_\beta
\right)(w_k)_\ell
\nonumber\\
&=
r\,a_r^{ij}
(\Gamma_{r,ij}^{\ell})_{,k}u_\ell
+
(\log A_r)_{,k}
-
\frac{r^4}{q_r}
(G_r^{\alpha\beta})_{,k}u_\alpha u_\beta,
\end{align*}
which is precisely \eqref{eq:linear-wk}. The preceding differentiation can be justified by applying the
difference-quotient method to \eqref{eq:det-Mr}, or alternatively
by first regularising the equation and then passing to the limit.

Note that Equation~\eqref{eq:linear-wk} can be further recast into the cofactor form: Let
\[
U_r^{ij}
:=
\operatorname{cof}(M_r)_{ij}
=
(\det M_r)(M_r^{-1})^{ij}
=
A_rq_r\,a_r^{ij}.
\]
Then
\begin{align}
U_r^{ij}(w_k)_{ij}
&+
\left(
-rU_r^{ij}\Gamma_{r,ij}^{\ell}
+
2r^4A_rG_r^{\ell \beta}u_\beta
\right)(w_k)_\ell
\nonumber\\
&=
rU_r^{ij}
(\Gamma_{r,ij}^{\ell})_{,k}u_\ell
+
q_r(A_r)_{,k}
-
r^4A_r(G_r^{\alpha\beta})_{,k}u_\alpha u_\beta.
\label{eq:cofactor-wk}
\end{align}

Recall from \eqref{eq:Mr-uniform-ellipticity} the uniform ellipticity of $M_r$. Moreover, in view of the identities:
\begin{align*}
    &(\Gamma_{r,ij}^{\ell})_{,k}(y)
=
r(\partial_{x_k}\Gamma_{ij}^{\ell})(ry),\qquad (G_r^{\alpha\beta})_{,k}(y)
 =
r(\partial_{x_k}g^{\alpha\beta})(ry),\\
&(K_r)_{,k}(y)
 =
r^{-1}(\partial_{x_k}K)(ry),\qquad(J_r)_{,k}(y)
 =
r(\partial_{x_k}\det g)(ry),\\
&(\log A_r)_{,k}
=
\frac{(K_r)_{,k}}{K_r}
+
\frac{(J_r)_{,k}}{J_r},
\end{align*}
we have 
\begin{equation*}
\left\|a_r^{ij}\right\|_{C^{0,\beta}(\mathcal A_1)}
+
\left\|b_r^\ell\right\|_{C^{0,\beta}(\mathcal A_1)}
+
\left\|f_{r,k}\right\|_{C^{0,\beta}(\mathcal A_1)}
\le C.
\end{equation*}
Therefore, applying the interior Schauder estimate to
\eqref{eq:linear-wk}, we obtain
\begin{equation*}
\|w_k\|_{C^{2,\beta}(\mathcal A_2)}
\le
C\left(
\|w_k\|_{C^0(\mathcal A_1)}
+
\|f_{r,k}\|_{C^{0,\beta}(\mathcal A_1)}
\right)
\le C.
\end{equation*}
Since $w_k=\partial_k u_r$, this proves that
\begin{equation}
\|u_r\|_{C^{3,\beta}(\mathcal A_2)}
\le C,
\label{eq:C3-annular}
\end{equation}
with $C$ independent of $r$.

\smallskip
\noindent
{\bf Step~6.}  We now scale back to the original variable $x$ with $r=|x|$ to obtain that
\begin{equation}
\left|Du(x)\right|\leq C|x|^2,\qquad \left|D^2u(x)\right|\leq C|x|,
\qquad
\left|D^3u(x)\right|\leq C
\label{eq:unscaled-estimates}
\end{equation}
for \emph{a.e.}  $x\neq0$ of sufficiently small modulus. In particular, the second derivatives $D^2u$ extend continuously across the origin with
\[
D^2u(0)=0.
\]
If the segment joining two points does not meet the origin, the
Lipschitz estimate for $D^2u$ follows by integrating $D^3u$ along
the segment. If the segment meets the origin, we split it at the
origin and use $|D^2u(x)|\leq C|x|$. In either case, we obtain the regularity $$u\in C^{2,1}.$$

This proves alternative \textup{(ii)} and completes the proof of the dichotomy Theorem~\ref{thm: dichotomy}.
\end{proof}

\begin{remark}
For \eqref{eq:unscaled-estimates}, we have actually established
the following scale-invariant bounds:
\begin{equation*}
\sup_{\mathcal A'_r}\left|D^3u\right|
+
r^\beta
\left[D^3u\right]_{C^\beta(\mathcal A'_r)}
\le C,
\end{equation*}
where $\mathcal{A}'_r = r \mathcal{A}_2$ is the concentric contraction of the annulus $\mathcal{A}_2$ by $r>0$, and $[\cdot]_{C^\beta}$ is the $\beta$-H\"{o}lder seminorm. No $C^{3,\beta}$-regularity at the degenerate point is asserted: indeed, the homogeneous model $u(x)=|x|^3$ (which is the standard local model exhibiting the same sharp
regularity as Iaia's example~\cite{iaia}) satisfies uniform
$C^{3,\beta}$-estimates after cubic rescaling on every fixed annulus 
but belongs only to $C^{2,1}$ at the origin.

\end{remark}


\section{$W^{3,p}$ isometric immersions for the degenerate Weyl problem}\label{sec: W3p}

In this section, we prove Theorem~\ref{thm:conditional-W3p}, reproduced below for convenience of the reader:
\begin{theorem*}
    
Let $(\Sigma,g)$ be a diffeomorphic $\stwo$ with $g\in W^{3,\infty}$, and let $(\M,\gbar)$ be a complete,  simply-connected 3-manifold with $\gbar \in W^{3,\infty}_\loc$. Suppose that the Gaussian curvature of $(\Sigma,g)$ satisfies $K_g \geq K_0$, and the sectional curvature of $(\M,\gbar)$ satisfies $\sec_{\overline g}\leq K_0$, where $K_0 \in \R$ is a constant. Assume that $\{g_\varepsilon\}_{0<\varepsilon<\varepsilon_0}$ is a family of positively curved $C^\infty$-metrics with uniformly bounded $W^{3,\infty}$-norms, such that $g_\e \to g$ in $\bigcap_{\alpha \in ]0,1[}C^{2,\alpha}(\Sigma)$, and that $\{g_\e\}$ admits $C^3$-isometric immersions $\{f_\varepsilon\}$ with images contained in a compact set $\mc \subset \M$ independent of $\varepsilon$.  

Then, under the assumptions that $\frac{1}{K_g-K_0}\in L^{\frac{p}{2}}(\Sigma,g)$ for some $p \geq 2$ and $\frac{k_\varepsilon}{H_\varepsilon}\geq\tau$ for some $\tau >0$ independent of $\e$ a.e. on $\Sigma$, there exists a subsequence of $\{f_\e\}$ that converges weakly in $W^{3,p}$ to an isometric immersion $f:(\Sigma,g)\to (\M,\gbar)$.
\end{theorem*}

Recall from~\eqref{small k, def} that $k := \sqrt{K_g-K_0}$. Then, for some $p \geq 2$, it is assumed that
    \begin{equation*}
   k^{-1} \in L^p(\Sigma,g).
    \end{equation*}
To prove the uniform $W^{3,p}$-regularity for $f_\varepsilon$, it suffices to establish the uniform $W^{1,p}$-bound for $H_\varepsilon$, the regularised mean curvatures. It is known that, by ellipticity of Codazzi equations, the gradient of mean curvature is controlled by the traceless part of the second fundamental form. The potential loss of the uniform ellipticity of the Codazzi equations is prevented by the ``pinching'' condition~\eqref{eq:uniform-global-pinching}:
$$\frac{k_\varepsilon}{H_\varepsilon}\geq\tau>0.$$ This enables us to bound $\|\na H_\e\|_{L^p} \leq C$ uniformly in $\e$ by a covering argument.

\subsection{A complex PDE for the traceless part of the second fundamental form}

Fix $x_0\in\Sigma$. Choose an isothermal coordinate $z=x+iy$ centred at $x_0$, in which  the uniform bounds in Lemma~\ref{lem:uniform-isothermal-charts} hold and $$
g_\varepsilon=e^{2u_\varepsilon}(dx^2+dy^2).$$ Thanks to the solution to the nondegenerate Weyl problem (\textit{cf. e.g.,} Nirenberg~\cite{N53}), $g_\e$ admits smooth isometric embeddings, whose second fundamental forms are labelled as 
$$\mathrm{II}_\varepsilon=L_\varepsilon\,dx^2+2M_\varepsilon\,dx\,dy+N_\varepsilon\,dy^2.
$$
We write 
\begin{align}
\Phi_\varepsilon =\frac{L_\varepsilon-N_\varepsilon}{2}-iM_\varepsilon,
\end{align}
which is nothing but $\mathring{\two}_\e : = \two_\e - {\rm trace}(\two_\e) I_{2 \times 2}$ in the complex coordinate $z$. Recall that $\partial_z=\frac12(\partial_x-i\partial_y)$ and $\partial_{\overline z}=\frac12(\partial_x+i\partial_y)$; meanwhile, since $H_\e$ is real-valued, $4|\partial_zH_\varepsilon|^2=|\partial_xH_\varepsilon|^2+|\partial_yH_\varepsilon|^2.$

Observe the first-order complex PDEs for $\Phi_\e$:
\begin{lemma}\label{lem: 1st order PDE for Phi}
The traceless part of the second fundamental form satisfies 
\begin{align}\label{eq:complex-Codazzi}
\partial_{\overline z}\Phi_\varepsilon=e^{2u_\varepsilon}\partial_zH_\varepsilon \underbrace{-\frac12(\mathcal R_{2,\varepsilon}+i\mathcal R_{1,\varepsilon})}_{\,=:\, \mathcal E_\varepsilon}
\end{align}
and
\begin{align}\label{eq:Hopf-standard-Beltrami}
&\partial_{\overline z}\Phi_\varepsilon-\frac{\overline{\Phi_\varepsilon}\partial_z\Phi_\varepsilon+\Phi_\varepsilon\partial_z\overline{\Phi_\varepsilon}}{2H_\varepsilon e^{2u_\varepsilon}} =\underbrace{\mathcal E_\varepsilon+\frac{e^{2u_\varepsilon}}{2H_\varepsilon}\partial_z(k_\varepsilon^2)-\frac{2e^{-2u_\varepsilon}(\partial_z u_\varepsilon)}{H_\varepsilon}|\Phi_\varepsilon|^2}_{=:\,\mathcal F_\varepsilon}.
\end{align}
Equation~\eqref{eq:Hopf-standard-Beltrami} is uniformly elliptic under the pinching condition~\eqref{eq:uniform-global-pinching}: $\frac{k_\varepsilon}{H_\varepsilon}\geq\tau>0$. The ellipticity coefficients depend only on $\tau$.  \end{lemma}

Here and hereafter, $\mathcal R_{i,\varepsilon}$, $i \in\{1,2\}$ denote the curvature coefficients:
\begin{align*}
&\mathcal R_{1,\varepsilon}=R^\M\bigl(df_\varepsilon(\partial_x),df_\varepsilon(\partial_y),df_\varepsilon(\partial_x),\nu_\varepsilon\bigr),\\
&\mathcal R_{2,\varepsilon}= R^\M\bigl(df_\varepsilon(\partial_x),df_\varepsilon(\partial_y),df_\varepsilon(\partial_y),\nu_\varepsilon\bigr).
\end{align*} 

\begin{proof}[Proof of Lemma~\ref{lem: 1st order PDE for Phi}]
It follows from the definition of $H_\varepsilon$ and the Gauss equation
\begin{equation}\label{eq:Hopf-Gauss}
\begin{cases}
H_\varepsilon=\frac{L_\varepsilon+N_\varepsilon}{2e^{2u_\varepsilon}},\\ k_\varepsilon^2=e^{-4u_\varepsilon}(L_\varepsilon N_\varepsilon-M_\varepsilon^2),\\
H_\varepsilon^2-k_\varepsilon^2=e^{-4u_\varepsilon}|\Phi_\varepsilon|^2.
\end{cases}
\end{equation}
Also, the Codazzi equations can be expressed as
\begin{equation}\label{xx, 0916}
\begin{cases} 
\partial_xM_\varepsilon-\partial_yL_\varepsilon+(\partial_yu_\varepsilon)(L_\varepsilon+N_\varepsilon)=\mathcal R_{1,\varepsilon},\\
\partial_xN_\varepsilon-\partial_yM_\varepsilon-(\partial_xu_\varepsilon)(L_\varepsilon+N_\varepsilon)=\mathcal R_{2,\varepsilon}.
\end{cases}
\end{equation}
Then Equation~\eqref{eq:complex-Codazzi} holds by applying $\partial_{\overline z}$ to $\Phi_\varepsilon$ and using Eqs.~\eqref{xx, 0916} and \eqref{eq:Hopf-Gauss}. On the other hand, differentiation of  the last identity in~\eqref{eq:Hopf-Gauss} yields that
\begin{equation*}
2H_\varepsilon\partial_zH_\varepsilon-\partial_z(k_\varepsilon^2)=e^{-4u_\varepsilon}\Big(\overline{\Phi_\varepsilon}\partial_z\Phi_\varepsilon+\Phi_\varepsilon\partial_z\overline{\Phi_\varepsilon}-4(\partial_zu_\varepsilon)|\Phi_\varepsilon|^2\Big).
\end{equation*} 
This together with the identity $\partial_zH_\varepsilon=e^{-2u_\varepsilon}(\partial_{\overline z}\Phi_\varepsilon-\mathcal E_\varepsilon)
$ leads to Equation~\eqref{eq:Hopf-standard-Beltrami}.

To see the uniform ellipticity of \eqref{eq:Hopf-standard-Beltrami}, we show that the principal part (\textit{i.e.}, the left-hand side) of the equation can be expressed in the form $\p_{\bar{z}}\Phi_\e + a\p_z\Phi_\e + b\p_z\overline{\Phi_\e}$ with $|a|+|b| <1$. Indeed, here $a=\bar{b} = \frac{\overline{\Phi_\e}}{2H_\e e^{2u_\e}}$, hence 
\begin{align*}
    |a|+|b| = \frac{|\Phi_\varepsilon|}{H_\varepsilon e^{2u_\varepsilon}}=\sqrt{1-\left(\frac{k_\varepsilon}{H_\varepsilon}\right)^2}\leq\sqrt{1-\tau^2}<1.
\end{align*}
\end{proof}

\subsection{Uniform $W^{1,p}$-estimate for $H_\e$: local case}
  
To proceed, we prove a local $\dot{W}^{1,p}$-estimate for the mean curvatures $H_\e$ over each geodesic ball on $\Sigma$ that may shrink to its centre $x_0$. We bound it by $k_\e(x_0)^{2/p}$ from the above, modulo some constant uniform in $\e$. In what follows, by writing integrals of the form $\int_{\B_\rho(x)}f\,\dd V_h$, we always understand the domain of integration as the geodesic ball with respect to $h$.

\begin{proposition}\label{prop:GC-arbitrary-p} 

Under the assumptions in Theorem~\ref{thm:conditional-W3p}, in particular, the pinching condition~\eqref{eq:uniform-global-pinching}, the following holds: Given any $\e>0$ and $x_0 \in \Sigma$, set $$r_0=r_*\min\left\{1,k_\varepsilon(x_0)^2\right\},$$ where $r_*>0$ is a small constant depending only on $p$, $\tau$, and the geometry of $\Sigma$ and $\M$. Then, for some constant $C_{p,\tau}>0$ independent of $\varepsilon$, $x_0$, and $\inf_\Sigma k_\varepsilon$, one has that
$$
\int_{\B_{2r_0}(x_0)}|\nabla H_\varepsilon|^p\,\dd V_{g_\varepsilon}\leq C_{p,\tau}k_\varepsilon(x_0)^pr_0^{2-p}.
$$
\end{proposition}

Note a crucial yet subtle point of the above result: $r_*$ is not only uniform in $\e$, but also independent of the position of $x_0$. This is crucial for the developments in the next subsection.

To elaborate on Proposition~\ref{prop:GC-arbitrary-p}, notice that for $k_\varepsilon(x_0)\leq1$, one has $r_0=r_*k_\varepsilon(x_0)^2$, so 
\begin{equation*}
\int_{\B_{2r_0}(x_0)}|\nabla H_\varepsilon|^p\,\dd V_{g_\varepsilon}\leq C_{p,\tau}k_\varepsilon(x_0)^pr_0^{2-p}\leq C_{p,\tau}\int_{\B_{8r_0}(x_0)}k_\varepsilon^{-p}\,\dd V_{g_\varepsilon},
\end{equation*}
thanks to $k_\e^{-p}(x) \approx k_\e^{-p}(x_0)$ for $x \in \B_{8r_0}(x_0)$; see \eqref{eq:GC-k-comparability}.  If $k_\varepsilon(x_0)\geq1$, then $r_0=r_*k$, so 
\begin{equation*}
\int_{\B_{2r_*}(x_0)}|\nabla H_\varepsilon|^p\,\dd V_{g_\varepsilon}\leq C_{p,\tau}k_\varepsilon(x_0)^pr_*^{2-p}\leq C_{p,\tau,r_*}.
\end{equation*}
The discussions above suggest that a uniform bound for the $L^p$-norm of $k_\e^{-1}$ near the degeneracy set $\{k=0\}$ plays an important role. 

\begin{proof}[Proof of Proposition~\ref{prop:GC-arbitrary-p}]

We divide our arguments into six steps below. 

\smallskip
\noindent
{\bf Step~1.} We first choose the scales. Let $r_{\mathrm{iso}}>0$ be as in Lemma~\ref{lem:uniform-isothermal-charts}, and let $C_\tau$ be as in Lemma~\ref{lem:uniform-bounds}. In particular, the geodesic ball $\B_{16 r_{\rm iso}}(x_0)$ lies in a single isothermal chart. Then we choose $r_*>0$ that satisfies $14r_*<r_{\mathrm{iso}}$ and $8C_\tau r_*\leq\frac14$, and set $r_0=r_*\min\{1,k_\varepsilon(x_0)^2\}$ as in the statement of this proposition. Thus, $$
\left|k_\varepsilon(x)^2-k_\varepsilon(x_0)^2\right|\leq8C_\tau r_0\leq8C_\tau r_*k_\varepsilon(x_0)^2\leq\frac14k_\varepsilon(x_0)^2 \qquad\text{for }x\in \B_{8r_0}(x_0),
$$
namely that
\begin{align}\label{eq:GC-k-comparability}
\frac{\sqrt3}{2}k_\varepsilon(x_0)\leq k_\varepsilon(x)\leq\frac{\sqrt5}{2}k_\varepsilon(x_0)\qquad\text{for }x\in \B_{8r_0}(x_0).
\end{align}
As $k_\e \leq H_\e \leq \tau^{-1}k_\e$ by~\eqref{eq:uniform-global-pinching} and the arithmetic mean-geometric mean inequality, one has
\begin{align}\label{eq:GC-H-comparability}
c_\tau\leq\frac{H_\varepsilon(x)}{k_\varepsilon(x_0)}\leq C_\tau\qquad\text{for }x\in \B_{8r_0}(x_0).
\end{align}

\smallskip
\noindent
{\bf Step~2.} In view of Step~1, let us rescale $H_\e$ and $\Phi_\e$ by $k_\e(x_0)^{-1}$ on $\B_{8r_0}(x_0)$. More precisely, consider the new coordinate 
\begin{equation*}
    \zeta := \frac{z}{r_0}
\end{equation*}
and write 
\begin{align*}
    B_s:=\left\{\frac{z(x)}{r_0}:\,x\in \B_{sr_0}(x_0)\right\},\qquad 0<s\leq8.
\end{align*}
Define on $B_8$ the rescaled quantities:
\begin{align*}
&\widetilde H_\varepsilon(\zeta):=\frac{H_\varepsilon(r_0\zeta)}{k_\varepsilon(x_0)} \geq c_\tau >0,\\ &\widetilde\Phi_\varepsilon(\zeta):=\frac{\Phi_\varepsilon(r_0\zeta)}{k_\varepsilon(x_0)},\\
&\widetilde u_\varepsilon(\zeta):=u_\varepsilon(r_0\zeta).
\end{align*} 
Recall the uniform estimates for the conformal factors of $g_\e$
Lemma~\ref{lem:uniform-isothermal-charts} on $\B_{8r_{\rm iso}}(x_0)$. This together with~\eqref{eq:GC-H-comparability} and the Gauss equation~\eqref{eq:Hopf-Gauss} leads to
\begin{align*}
\left\|\widetilde u_\varepsilon\right\|_{C^{0,1}(B_8)} + \left\|\widetilde H_\varepsilon\right\|_{L^\infty(B_8)} + \left\|\widetilde\Phi_\varepsilon\right\|_{L^\infty(B_8)} \leq C_\tau.
\end{align*}

The equations for $\widetilde\Phi_\varepsilon$ are easily derived by rescaling the PDEs in Lemma~\ref{lem: 1st order PDE for Phi}. Indeed, from Equation~\eqref{eq:Hopf-standard-Beltrami} we infer that 
\begin{equation}\label{eq:GC-rescaled-Beltrami}
\partial_{\overline\zeta}\widetilde\Phi_\varepsilon-\frac{\overline{\widetilde\Phi_\varepsilon}\partial_\zeta\widetilde\Phi_\varepsilon+\widetilde\Phi_\varepsilon\partial_\zeta\overline{\widetilde\Phi_\varepsilon}}{2\widetilde H_\varepsilon e^{2\widetilde u_\varepsilon}}=\frac{r_0}{k_\varepsilon(x_0)}\mathcal F_\varepsilon(r_0\zeta),
\end{equation} 
and from the Codazzi Equation~\eqref{eq:complex-Codazzi} we infer that
\begin{align}\label{eq:GC-rescaled-Codazzi}
e^{2\widetilde u_\varepsilon}\partial_\zeta\widetilde H_\varepsilon&=\partial_{\overline\zeta}\widetilde\Phi_\varepsilon-\frac{r_0}{k_\varepsilon(x_0)}\mathcal E_\varepsilon(r_0\zeta).
\end{align}

\smallskip
\noindent
{\bf Step~3.} Now we check that the right-hand sides of Equations~\eqref{eq:GC-rescaled-Beltrami} and \eqref{eq:GC-rescaled-Codazzi} are bounded. Indeed, by Lemmata~\ref{lem:uniform-isothermal-charts} and~\ref{lem:uniform-bounds}, $|\partial_z(k_\varepsilon^2)|\leq C_\tau$ while $e^{2u_\varepsilon}$,  $|\partial_zu_\varepsilon|$, and $\mathcal{E}_\e$ are uniformly bounded. Hence, 
\begin{align*}
\left|\frac{e^{2u_\varepsilon}}{2H_\varepsilon}\partial_z(k_\varepsilon^2)\right| \leq C_\tau k_\varepsilon^{-1}.
\end{align*}
In addition, taking into account the condition~\eqref{eq:uniform-global-pinching} and uniform upper bound for $k_\varepsilon$, we have
\begin{equation*}
\left|\frac{2e^{-2u_\varepsilon}\partial_zu_\varepsilon}{H_\varepsilon}\right| |\Phi_\varepsilon|^2=2e^{2u_\varepsilon}|\partial_zu_\varepsilon|\left(\frac{H_\varepsilon^2-k_\varepsilon^2}{H_\varepsilon}\right)\leq C_\tau.
\end{equation*}
Thus, in view of the definition of $\mathcal{F}_\e$ and the bound~\eqref{eq:GC-k-comparability},
\begin{equation}\label{eq:Hopf-source-bound}
\left|\mathcal F_\varepsilon(x)\right|\leq C_\tau\left(1+\frac{1}{k_\varepsilon(x)}\right) \leq C_\tau\left(1+\frac{1}{k_\varepsilon(x_0)}\right)\qquad\text{for } x \in \B_{8r_0}(x_0).
\end{equation}
This together with $|\mathcal{E}_\e| \leq C$ implies that \begin{equation}\label{eq:GC-rescaled-source}
\left\|\frac{r_0}{k_\varepsilon(x_0)}\mathcal F_\varepsilon(r_0\,\cdot)\right\|_{L^\infty(B_8)}+\left\|\frac{r_0}{k_\varepsilon(x_0)}\mathcal E_\varepsilon(r_0\,\cdot)\right\|_{L^\infty(B_8)}\leq C_\tau.
\end{equation}

\smallskip
\noindent
{\bf Step~4.} We are now at the staging of applying elliptic estimates to Equations~\eqref{eq:GC-rescaled-Beltrami} and \eqref{eq:GC-rescaled-Codazzi} derived in Step~2, as all the coefficients in these PDEs (which are uniformly elliptic by Lemma~\ref{lem: 1st order PDE for Phi}) have been shown to lie in $L^\infty$.

Indeed, applying \textit{e.g.}, Meyers~\cite[Theorem~2]{Meyers} to \eqref{eq:GC-rescaled-Beltrami}, we deduce that 
$$
\left\|\widetilde\Phi_\varepsilon\right\|_{W^{1,q_0}(B_6)}\leq C_\tau\left(\left\|\widetilde\Phi_\varepsilon\right\|_{L^2(B_8)}+\left\|\frac{r_0}{k_\varepsilon(x_0)}\mathcal F_\varepsilon(r_0\,\cdot)\right\|_{L^{q_0}(B_8)}\right)\leq C_\tau\qquad\text{for some }q_0 = q_0(\tau) >2. 
$$
In particular, it implies $\left\|\partial_{\overline\zeta}\widetilde\Phi_\varepsilon\right\|_{L^{q_0}(B_6)} \leq C_\tau$, so the right-hand side of Equation~\eqref{eq:GC-rescaled-Codazzi} is uniformly bounded in $L^{q_0}$, in view of \eqref{eq:GC-rescaled-Codazzi} once again. The same arguments then yield that 
$$
\left\|\widetilde H_\varepsilon\right\|_{W^{1,q_0}(B_6)}\leq C_\tau.
$$
In addition, by Sobolev--Morrey embedding in dimension two,
\begin{align} \label{eq:GC-first-Holder}
\left\|\widetilde H_\varepsilon\right\|_{C^{0,\alpha}(B_5)} + \left\|\widetilde\Phi_\varepsilon\right\|_{C^{0,\alpha}(B_5)} \leq C_\tau \qquad\text{for } \alpha=1-\frac2{q_0}>0.
\end{align}

\smallskip
\noindent
{\bf Step~5.} The $W^{1,q_0}$-estimate for $\widetilde{H}_\e$ obtained in Step~4 above (where $q_0$ is possibly smaller than $p$) can be easily promoted to a $W^{1,p}$-estimate, by virtue of the H\"{o}lder bound~\eqref{eq:GC-first-Holder}.

Indeed, the positive lower bound for $\widetilde{H}_\e$ (which follows from \eqref{eq:GC-H-comparability}), the Lipschitz bound for $\widetilde{u}_\e$, and~\eqref{eq:GC-first-Holder} together imply that 
$$
\left\|\frac{\widetilde\Phi_\varepsilon}{\widetilde H_\varepsilon e^{2\widetilde u_\varepsilon}}\right\|_{C^{0,\alpha}(B_5)}\leq C_\tau.
$$
From the same elliptic estimates as in Step~4 above applied to \eqref{eq:GC-rescaled-Beltrami}, we deduce that 
$$
\left\|\widetilde\Phi_\varepsilon\right\|_{W^{1,p}(B_4)}\leq C_{p,\tau},
$$
with the constant independent of $\varepsilon$ and $k_\varepsilon(x_0)$. Substituting into~\eqref{eq:GC-rescaled-Codazzi} yields, by the elliptic estimates once again, that
\begin{align}
\left\|\widetilde H_\varepsilon\right\|_{W^{1,p}(B_3)} \leq C_{p,\tau}.
\label{eq:GC-rescaled-W1p-H}
\end{align}

\smallskip
\noindent
{\bf Step~6.} Finally, the desired estimate $
\int_{\B_{2r_0}(x_0)}|\nabla H_\varepsilon|^p\,\dd V_{g_\e}\leq C_{p,\tau}k_\varepsilon(x_0)^pr_0^{2-p}$ follows from scaling back to the original variable $z=x+iy$. The proof of Proposition~\ref{prop:GC-arbitrary-p} is now complete.    \end{proof}

\begin{remark} \label{rmk Labourie eq}
Proposition~\ref{prop:GC-arbitrary-p} can also be proved using the J-holomorphic curve formulation of the Gauss--Codazzi equations \textit{\`{a} la} Labourie~\cite{Labourie}; in particular, the PDE~\eqref{PDE for H} for $H_\e$ or~\eqref{PDE for W} for $H_{\e}^{-1}$. Here we present an alternative method using the Beltrami-type system, which has the advantange of treating the nearly umbilic region ($\kappa_{1,\e}\approx \kappa_{2,\e}$) and its complement at one strike. 
\end{remark}

\subsection{Uniform $W^{1,p}$-estimate for $H_\e$: global case} In this subsection, we shall derive a uniform bound for the $W^{1,p}$-norm of $H_\e$ over the whole domain $\Sigma$, by globalising the local estimates over small geodesic balls $\B_{2r_0(x_0)}^{g_\e}$ obtained in Proposition~\ref{prop:GC-arbitrary-p}. Such small balls degenerate near the  near the degeneracy set $\{k=0\}\cap\Sigma$: indeed, $r_0 \approx k_\e^2(x_0) \to 0^+$ as $\e \to 0^+$.

\begin{proposition}\label{prop:global-H-W1p}
Under the assumptions in Theorem~\ref{thm:conditional-W3p}, we have that
\begin{align*}
\|H_\varepsilon\|_{W^{1,p}(\Sigma,g_\varepsilon)}\leq C,
\end{align*}
where $C$ depends only on $p$, $\tau$ (the positive lower bound for $k_\e/H_\e$), and the geometries of $(\Sigma,g)$ and $(\M,\gbar)$.
\end{proposition}

The pinching condition~\eqref{eq:uniform-global-pinching} plays a crucial role in the above proposition. The constant $C$ is independent of $\varepsilon$ and $\inf_\Sigma k_\varepsilon$.

\begin{proof}[Proof of Proposition~\ref{prop:global-H-W1p}]

Recall that a uniform $L^\infty$-bound has been established by Lemma~\ref{lem:uniform-bounds}, so it suffices to prove a bound for $\int_\Sigma |\na H_\e|^p\,\dd V_{g^\e}$.  Throughout the proof, all geodesic balls on $\Sigma$ are taken with respect to $g_\e$, unless otherwise specified.

Let $r_* = r_*(p,\tau,\Sigma,\M)>0$ be chosen as in Proposition~\ref{prop:GC-arbitrary-p}. Define $r_\e: \Sigma \to [0,\infty[$ by
\begin{equation*}
    r_\varepsilon(x):=r_*\min\big\{1,k_\varepsilon(x)^2\big\}.
\end{equation*}
By virtue of Lemma~\ref{lem:uniform-bounds} and the choice of $r_*$, we have
\begin{align}\label{eq:radius-function-Lipschitz}
|r_\varepsilon(x)-r_\varepsilon(y)|\leq\frac{1}{32}d_{g_\varepsilon}(x,y)
\qquad\text{for all }x,y\in\Sigma.
\end{align}
The function $r_\varepsilon$ is positive and continuous on the compact surface $\Sigma \cong\stwo$ for each given $\varepsilon>0$.

Next, for each fixed $\varepsilon>0$, let us choose a maximal collection of points $\mathscr{C}_\e:=\left\{x_{i,\varepsilon}\right\}_{1 \leq i \leq N_\e}\subset\Sigma$ such that the geodesic balls $\B_{r_{i,\varepsilon}/20}(x_{i,\varepsilon})$ are pairwise disjoint. Here we put $$r_{i,\varepsilon}:=r_\varepsilon(x_{i,\varepsilon})=r_*\min\big\{1,k_\varepsilon(x_{i,\varepsilon})^2\big\}.$$ Observe that $N_\e<\infty$ by the continuity and positivity of $r_\e$ and the compactness of $\Sigma$. Maximality of the set of centres $\mathscr{C}_\e$ and the bound~\eqref{eq:radius-function-Lipschitz} imply that
$$\Sigma\subset\bigcup_{i=1}^{N_\varepsilon}\B_{\frac{r_{i,\varepsilon}}{4}}(x_{i,\varepsilon}).$$

Let us \emph{claim} that
\begin{enumerate}
    \item[(i)]
The collection of balls $\mathscr{B}_\e:=\left\{\B_{8r_{i,\varepsilon}}(x_{i,\varepsilon})\right\}_{1 \leq i \leq N_\e}$ is uniformly locally finite: There exists $N_{1,*}$ independent of $\varepsilon$ and $\inf_\Sigma k_\varepsilon$ such that every $x\in\Sigma$ is contained in at most $N_{1,*}$ members of $\mathscr{B}_\e$.
\item[(ii)]
The collection of points $x_{i,\e}$ in $\mathscr{C}_\e$ such that $k_{\e}(x_{i,\e}) \geq 1$ (equivalently, $r_{i,\e} \equiv r_\e(x_{i,\e})=r_*$) has less than $N_{2,*}$ elements, where $N_{2,*}$ is independent of $\varepsilon$ and $\inf_\Sigma k_\varepsilon$.
    
\end{enumerate}

\begin{proof}[Proof of the claim]
Indeed, for (i), observe that the intersecting balls in $\mathscr{B}_\e$ have comparable radii: For any $y \in \Sigma$ we set  $I_\varepsilon(y):=\{i:y\in \B_{8r_{i,\varepsilon}}(x_{i,\varepsilon})\}$, then $
d_{g_\varepsilon}(x_{i,\varepsilon},x_{j,\varepsilon})\leq8\bigl(r_{i,\varepsilon}+r_{j,\varepsilon}\bigr)$ for $i,j\in I_\varepsilon(y)$, which by Equation~\eqref{eq:radius-function-Lipschitz} leads to $\frac35r_{j,\varepsilon}\leq r_{i,\varepsilon}\leq\frac53r_{j,\varepsilon}$. Now, fix an arbitrary $i_0\in I_\varepsilon(y)$. The above arguments and the construction of $\mathscr{B}_\e$ yield that, for any $i\in I_\varepsilon(y)$:
\begin{itemize}
    \item 
$\B_{r_{i,\varepsilon}/20}(x_{i,\varepsilon})$ are pairwise disjoint;
\item 
$\B_{r_{i,\varepsilon}/20}(x_{i,\varepsilon})$ have radii bounded below by $3r_{i_0,\varepsilon}/100$; and
\item 
$\B_{r_{i,\varepsilon}/20}(x_{i,\varepsilon}) \subset \B_{14r_{i_0,\varepsilon}}(y)$. 
\end{itemize}
Thus, for constants $c$ and $C$ depending only on the volume forms $V_{g^\e}$ restricted to these balls (which, by Proposition~\ref{prop:curvature-compatible-approximation}, are in turn controlled by constants uniform in $\e$ and depending only on the $C^2$-geometry of $(\Sigma,g)$), we have 
\begin{align*}
\#I_\varepsilon(y)\cdot c\cdot\left(\frac{3r_{i_0,\varepsilon}}{100}\right)^2 \leq\sum_{i\in I_\varepsilon(y)}V{g_\varepsilon}\bigl(\B_{r_{i,\varepsilon}/20}(x_{i,\varepsilon})\bigr) \leq V_{g_\varepsilon}\bigl(\B_{14r_{i_0,\varepsilon}}(y)\bigr) \leq C \cdot \bigl(14r_{i_0,\varepsilon}\bigr)^2.
\end{align*}
It follows that $\# I_\e(y) \leq \frac{C}{c} \cdot \frac{1400^2}{9} =: N_{1,*}$. 

As for (ii), if $k_\e(x_{i,\e}) \geq 1$ with $x_{i,\e} \in \mathscr{C}_\e$, then each of the pairwise disjoint balls $\B_{r_{i,\e}/20}(x_{i,\e})$ has radius $\geq r_*/20$, where the parameter $r_*>0$ depends only on $p$, $\tau$, and the geometry of $\Sigma$ and $\M$ by Proposition~\ref{prop:GC-arbitrary-p}. Thus, the area of each of these balls is bounded from below by a positive constant that is independent of $\e$, $\inf_{\e}{k_\e}$, and the position of the centres. But $(\Sigma,g)$ has finite area, so the number of such balls is uniformly bounded from above.    \end{proof}

Now we can conclude the global estimate for $\|H_\e\|_{\dot{W}^{1,p}}$. Recall from  Proposition~\ref{prop:GC-arbitrary-p}:
\begin{equation*}
    \int_{\B_{2r_0}(x_0)}|\nabla H_\varepsilon|^p\,\dd V_{g_\varepsilon}\leq \begin{cases}
         C_{p,\tau}\int_{\B_{8r_0}(x_0)}k_\varepsilon^{-p}\,\dd V_{g_\varepsilon}\qquad\text{ if $k_\varepsilon(x_0)\leq1$},\\
          C_{p,\tau,r_*} \qquad\text{if $k_\varepsilon(x_0)\geq1$}.
    \end{cases}
\end{equation*}
Since the balls $\left\{\B_{2r_{i,\varepsilon}}(x_{i,\varepsilon})\right\}_{1\leq i \leq N_\e}$ cover $\Sigma$, we estimate that
\begin{align*}
\int_\Sigma|\nabla H_\varepsilon|^p\,\dd V_{g_\varepsilon}
&\leq \left(\sum_{\{i: r_{i,\e}<r_*\}} + \sum_{\{i: r_{i,\e}= r_*\}}\right)\int_{\B_{2r_{i,\varepsilon}}(x_{i,\varepsilon})}|\nabla H_\varepsilon|^p\,\dd V_{g_\varepsilon}\\
&\leq C'_{p,\tau}N_{1,*} \int_\Sigma k_\e^{-p}\,\dd V_{g^\e} + C'_{p,\tau,r_*} N_{2,*}\\
&\leq C''_{p,\tau}N_{1,*} \int_\Sigma k^{-p}\,\dd V_{g^\e} + C'_{p,\tau,r_*} N_{2,*},
\end{align*}
where the final line follows from Lemma~\ref{lem:inverse-curvature-transfer}. The right-hand side is independent of $\e$ and $\inf_\Sigma k_\e$.

The proof of Proposition~\ref{prop:global-H-W1p} is now complete.     \end{proof}

\subsection{Proof of the $W^{3,p}$-isometric immersion}

In this subsection, we pass from the $W^{1,p}$-estimate for $H_\e$ to the $W^{3,p}$-estimate for the isometric immersions, thus establishing Theorem~\ref{thm:conditional-W3p}. 

The traceless part of the second fundamental form associated with $f_\e$ is 
$$
\mathring{\mathrm{II}}_\varepsilon := \mathrm{II}_\varepsilon - H_\varepsilon g_\varepsilon.
$$
We shall make use of the following lemma:

\begin{lemma}\label{lem: ellipticity}
When viewed as a PDE over the symmetric traceless 2-tensors $\mathring{\mathrm{II}}_\varepsilon$, the Codazzi equations form an elliptic system, which degenerates only at the zero solution.
\end{lemma}

\begin{proof}[Proof of Lemma~\ref{lem: ellipticity}]
The Codazzi equations are equivalent to
\begin{align} \label{eq:trace-free-Codazzi}
(\nabla_i\mathring{\mathrm{II}}_\varepsilon)_{jk} -(\nabla_j\mathring{\mathrm{II}}_\varepsilon)_{ik} = R^\M \bigl(df_\varepsilon(\partial_i),df_\varepsilon(\partial_j), df_\varepsilon(\partial_k),\nu_\varepsilon \bigr) -(\nabla_iH_\varepsilon)(g_\varepsilon)_{jk} +(\nabla_jH_\varepsilon)(g_\varepsilon)_{ik}.
\end{align}

Write in any $g_\varepsilon$-orthonormal frame $
\mathring{\mathrm{II}}_\varepsilon = \begin{pmatrix} a&b\\ b&-a \end{pmatrix}
$. The principal symbol for the differential operator on the left-hand side of~\eqref{eq:trace-free-Codazzi} is $
(a,b) \mapsto \begin{pmatrix} \xi_1b-\xi_2a\\ -\xi_1a-\xi_2b \end{pmatrix}
$ for $\xi = (\xi_1, \xi_2)^\top \in \R^2$. The lemma follows from the identity $|\xi_1b-\xi_2a|^2+|-\xi_1a-\xi_2b|^2=(\xi_1^2+\xi_2^2)(a^2+b^2)=\frac12|\xi|^2|\mathring{\mathrm{II}}_\varepsilon|^2$.   \end{proof}

\begin{proof}[Proof of Theorem~\ref{thm:conditional-W3p}]

Let $\{g_\e\}$ be the family of smooth metrics with strictly positive Gaussian curvature approximates $g$ in $C^{2,\alpha}$ for every $0<\alpha<1$, let $f_\e$ be the $\e$-elliptic $C^3$-isometric immersion, and let $\tau>0$ be a lower bound for the ratio $k_\e/H_\e$, all as in the assumptions of this theorem. Again, unless otherwise specified, geodesic balls on $\Sigma$ are taken with respect to $g_\e$.

We divide our proof into three steps below.

\smallskip
\noindent
{\bf Step~1.} It suffices to show that for given $\varrho>0$ (independent of $\e$) and $x_0 \in \Sigma$, we have 
\begin{equation}\label{to show, local}
    \|f_\e\|_{W^{3,p}(\B_\varrho(x_0))} \leq C,
\end{equation}
where $C$ depends only $p$, $\tau$, $\varrho$, $\Sigma$, and $\M$, as well as an upper bound for $\|H_\e\|_{W^{1,p}\left(\B^{g_\e}_{2\varrho}(x_0)\right)}$. Indeed, we may choose $\varrho>0$ so small that  for any $x_0 \in \Sigma$, the union of geodesic balls $\bigcup_\e\left[\B^{g_\e}_{10^2 \varrho}(x_0)\right]$ lies in one isothermal coordinate chart. We have established in  Proposition~\ref{prop:global-H-W1p} that $\|H_\varepsilon\|_{W^{1,p}(\Sigma,g_\varepsilon)}\leq \Lambda$, where $\Lambda$ depends only on $p$, $\tau$, and the geometries of $\Sigma$ and $\M$. Thus, once~\eqref{to show, local} is proved for some $C=C(p,\varrho,\Sigma,\M,\Lambda)$, as $\Sigma$ is compact and $g_\e$ uniformly approximates $g$ by Proposition~\ref{prop:curvature-compatible-approximation}, a simple covering argument leads to $\|f_\e\|_{W^{3,p}(\Sigma,g_\e)} \leq C_{p,\tau,\Sigma,\M}$. This completes the proof.

\smallskip
\noindent
{\bf Step~2.} To show~\eqref{to show, local}, let us first notice that $$\|f_\e\|_{W^{1,\infty}(\B_{2\varrho(x_0)})} \leq C_{\varrho,\Sigma,\M},$$ since $f_\e$ are isometric immersions of uniformly $W^{3,\infty}$-bounded metrics $g_\e$, and the images of $f_\e$ are confined in a fixed compact subset of $\M$. Then, by applying the elliptic estimates to the traceless part of the second fundamental form $\mathring{\two}_\e$, we deduce that
\begin{align*}
\left\|\na \mathring{\two}_\e\right\|_{L^p\left(\B_{\varrho}(x_0)\right)} &\leq C\left(1 + \left\|\mathring{\two}_\e\right\|_{L^p\left(\B_{2\varrho}(x_0)\right)} + \|\na H_\e\|_{L^p\left(\B_{2\varrho}(x_0)\right)}\right)\\
&\leq C \left(1+ \|H_\e\|_{W^{1,p}\left(\B_{2\varrho}(x_0)\right)}\right)
\end{align*}
with the constant $C$ depending on $\varrho$, $p$, $\tau$, as well as the geometries of $\Sigma$ and $\M$. Thus, in view of Proposition~\ref{prop:curvature-compatible-approximation}, we conclude that
\begin{align*}
\left\|{\two}_\e\right\|_{W^{1,p}\left(\B_{\varrho}(x_0)\right)} \leq C_{p,\tau,\varrho,\M,\Sigma}.
\end{align*}
By the uniform $W^{2,\infty}$-boundedness of $g_\e$ and $g_\e^{-1}$, the same estimate also holds for $S_\e$, the shape operator associated with $\two_\e$, defined by $(S_\varepsilon)_i^j =(g_\varepsilon)^{jk}(\mathrm{II}_\varepsilon)_{ik}$.

\smallskip
\noindent
{\bf Step~3.} Let $\Gamma_{ij}^m(g_\varepsilon)$ and $\overline\Gamma_{\beta\gamma}^\alpha(\overline g)$ be the Christoffel symbols associated with $g_\varepsilon$ and $\overline g$, respectively. The Gauss and Weingarten equations read:
\begin{align}
\partial_{ij}f_\varepsilon^\alpha ={}& \Gamma_{ij}^m(g_\varepsilon)\partial_mf_\varepsilon^\alpha -\overline\Gamma_{\beta\gamma}^\alpha(f_\varepsilon) \partial_if_\varepsilon^\beta \partial_jf_\varepsilon^\gamma +(\mathrm{II}_\varepsilon)_{ij}\nu_\varepsilon^\alpha, \label{eq:coordinate-Gauss}\\
\partial_i\nu_\varepsilon^\alpha ={}& -(S_\varepsilon)_i^m\partial_mf_\varepsilon^\alpha -\overline\Gamma_{\beta\gamma}^\alpha(f_\varepsilon) \partial_if_\varepsilon^\beta \nu_\varepsilon^\gamma. \label{eq:coordinate-Weingarten}
\end{align}
In this step, all the constants $C_0$ are independent of $\e$.

Thanks the uniform boundedness of $\two_\e$ in $W^{1,p}$ and $f_\e$ in $W^{1,\infty}$, the $W^{2,\infty}$-boundedness of Christoffel symbols, as well as the continuous embedding $W^{1,p}(\Sigma)\emb L^q(\Sigma)$ for any $q < \infty$ and $p \geq 2$, and in light of~\eqref{eq:coordinate-Gauss}, we deduce that $\|f_\e\|_{W^{2,q}(\Sigma)} \leq C_0$ for any $q<\infty$. Then, by \eqref{eq:coordinate-Weingarten} we have that $\|\nu_\e\|_{W^{1,q}}\leq C_0$ for any $q<\infty$. Substituting this back to the Gauss Equation~\eqref{eq:coordinate-Gauss}, we find that its right-hand side is uniformly bounded in $W^{1,p}$, so we arrive at \eqref{to show, local}.

The proof of Theorem~\ref{thm:conditional-W3p} is now complete.   \end{proof}

\section{The strictly elliptic generalised Weyl problem}\label{sec: strictly elliptic}
In this section, we prove Theorem~\ref{thm:strict-C21-Weyl}, which asserts that a metric $g\in C^{2,1}$ with strictly positive relative Gaussian curvature $K_g-K_0\geq\delta_0>0$ admits an isometric immersion into $(\M,\gbar)$ of class $W^{3,p}$ for every finite $p$.

\subsection{$L^\infty$-estimate for the second fundamental form}

\subsubsection{Approximation}
Let $g \in C^{2,1}$ be a metric with  strictly positive relative Gaussian curvature. We first approximate it by a sequence of smooth metrics admitting strictly elliptic isometric immersions $f_\e : (\Sigma,g) \to (\M,\gbar)$. Labourie~\cite{Labourie} ensures that this can be done in a ``non-wandering'' manner, even if $\M$ is noncompact.

Indeed, let $\{g_\varepsilon\}_{0<\varepsilon<\varepsilon_0}$ be the metrics constructed in Proposition~\ref{prop:curvature-compatible-approximation}. Since $K_g\geq K_0+\delta_0$ and $g_\varepsilon\to g$ in $C^{2,\alpha}$, we have  $K_{g_\varepsilon}-K_0\geq\frac{\delta_0}{2}$ by shrinking $\varepsilon_0$ if necessary. For each $x_0\in\Sigma$, $y_0\in \M$, an oriented 2-dimensional subspace $P_0\subset T_{y_0}M$, and an orientation-preserving linear isometry $
I_\varepsilon: (T_{x_0}\Sigma,g_\varepsilon(x_0)) \longrightarrow (P_0,\overline g(y_0))$, by~\cite[Th\'eor\`eme~B and \S7]{Labourie} there exists a unique $C^\infty$-orientation-preserving strictly elliptic isometric immersion 
\begin{align} \label{eq:strict-normalized-embeddings}
&f_\varepsilon:(\Sigma,g_\varepsilon) \longrightarrow(M,\overline g)\quad \text{ such that }\quad
f_\varepsilon(x_0)=y_0,\,\, df_\varepsilon|_{x_0}=I_\varepsilon,\nonumber\\
&\text{and each $f_\varepsilon$ is an embedding onto the boundary of a convex subset of $\M$.}
\end{align}

Set $$R_0:=\sup_{0<\varepsilon<\e_0}\operatorname{diam}_{g_\varepsilon}(\Sigma)<\infty.$$ Since $f_\varepsilon(x_0)=y_0$ and $f_\varepsilon$ is isometric, one has $d_{\overline g}(y_0,f_\varepsilon(x))\leq d_{g_\varepsilon}(x_0,x)\leq R_0$ for every $x\in\Sigma$. The geodesic ball  $\overline \B^{\overline g}_{R_0}(y_0)$ is compact and contains every $f_\varepsilon(\Sigma)$. We may thus apply the isothermal coordinates in \S\ref{subsec: unif bd} to $\mc:=\overline \B^{\overline g}_{R_0}(y_0)$. In particular,
\begin{align} \label{eq:strict-relative-lower-bound}
\frac{\delta_0}{2} \leq K_{g_\varepsilon}-K_0 &\leq k_\varepsilon^2 := K_{g_\varepsilon} -\sec_{\overline g}\bigl(df_\varepsilon(T\Sigma)\bigr)\nonumber\\
&\leq \|K_{g_\varepsilon}\|_{L^\infty(\Sigma)} + \sup_{y\in\mc}\bigl|\sec_{\overline g}(y)\bigr| \leq C.
\end{align} 

\subsubsection{Bound for total mean curvature}
By, \emph{e.g.}, Labourie~\cite[Proposition~5.4(iii) and the final paragraph of the proof of Th\'eor\`eme~A]{Labourie} and also \cite{h60,Schulz90,LuWeyl}, we have the Minkowski formula: 
\begin{proposition}\label{prop:strict-total-mean-curvature}
The mean curvatures $H_\varepsilon$ of the isometric immersions $f_\varepsilon$ in \eqref{eq:strict-normalized-embeddings} satisfy
\begin{align}\label{eq:strict-total-mean-curvature}
\sup_\varepsilon\int_\Sigma H_\varepsilon\,\dd V_{g_\varepsilon}\leq M_0,
\end{align}
where $M_0$ is independent of $\varepsilon$.
\end{proposition}

\subsubsection{$L^\infty$-bound for $\two_\e$} For strictly elliptic isometric immersions, an $L^1$-bound for the mean curvature implies an $L^\infty$-bound for the second fundamental form. Lu's arguments in~\cite{LuWeyl} may be directly adapted to yield the following:

\begin{proposition} \label{prop:strict-H-and-pinching}
For the isometric immersions $f_\varepsilon$ as in \eqref{eq:strict-normalized-embeddings}, there are uniform  constants $\Lambda_0, c_1,C_1,\tau_0>0$ and $\alpha \in ]0,1[$ (all independent of $\e$ and $\inf_\Sigma k_\e$), such that
\begin{align*}
&\sup_{0<\varepsilon<\e_0} \left\{\|S_\varepsilon\|_{C^{0,\alpha}(\Sigma)}+\|H_\varepsilon\|_{C^{0,\alpha}(\Sigma)}+\left\|\frac{\kappa_{2,\varepsilon}-\kappa_{1,\varepsilon}}{\kappa_{1,\varepsilon}+\kappa_{2,\varepsilon}}\right\|_{C^{0,\alpha}(\Sigma)}\right\} \leq \Lambda_0,\\
&0<c_1\leq\kappa_{1,\varepsilon}\leq\kappa_{2,\varepsilon}\leq C_1\quad\text{and}\quad\frac{k_\varepsilon}{H_\varepsilon}\geq\tau_0 \qquad\text{on }\Sigma.
\end{align*} 
\end{proposition}

The norms are taken with respect to  $g_\e$, and the relevant bounds are all uniform in $\e$. 

\begin{proof}[Proof of Proposition~\ref{prop:strict-H-and-pinching}]
From the Gauss equation, we have 
\begin{equation}
\label{eq:strict-coordinate-Gauss-determinant}
\det\bigl((\mathrm{II}_\varepsilon)_{ab}\bigr) = k_\varepsilon^2 \det\bigl((g_\varepsilon)_{ab}\bigr) \geq c_0>0,
\end{equation}
where $c_0$ is independent of $\varepsilon$. The coordinate metrics are uniformly elliptic and have uniform $C^{2,1}$-bounds in each fixed chart. This together with Proposition~\ref{prop:strict-total-mean-curvature} verifies the hypotheses for Lu's interior estimates in~\cite[Proposition~4.1, Lemma~2.5, and the proof of Theorem~4.2]{LuWeyl}, which yields the uniform $L^\infty$-bounds for $\two_\e$.

Then, note that $\kappa_{2,\varepsilon}\leq2H_\varepsilon\leq2\Lambda_0$. Since $\kappa_{1,\varepsilon}\kappa_{2,\varepsilon}=k_\varepsilon^2$ and $k_\varepsilon^2\geq\delta_0/2$, we thus have $\kappa_{1,\varepsilon}\geq\delta_0/(4\Lambda_0)$. Meanwhile, using the estimate $H_\varepsilon\leq\Lambda_0$ just obtained, we deduce $k_\varepsilon/H_\varepsilon\geq\sqrt{\delta_0/2}/\Lambda_0$. This proves the final line.

It remains to promote the $L^\infty$-bound for $\two_\e$ to the $C^{0,\alpha}$-estimates for the shape operator and the principal curvature gap. This follows from the ``Heinz--Lewy argument'' as in Lu~\cite[proof of Theorem~4.2, in particular (4.11)]{LuWeyl}. Indeed, the above cited result estimates $\|S_\e\|_{C^{0,\alpha}(\Sigma)}$ for some $\alpha \in ]0,1[$ by a constant depending only on the determinant lower bound~\eqref{eq:strict-coordinate-Gauss-determinant}, the bounds for $g_\varepsilon$ and $\overline g$ on $\mc$, the total mean curvature estimate~\eqref{eq:strict-total-mean-curvature}, and the pointwise lower bound for $\mathrm{II}_\varepsilon$; but all these quantities have been proved to be uniform in $\e$. Moreover, as ${\kappa_{1,\e}}$ and  $\kappa_{2,\e}$ are eigenvalues of the symmetric matrix field $S_\e$, and their sum (which is $2H_\e$) have been proved bounded from below, so the $C^{0,\alpha}$-estimate for $\frac{\kappa_{2,\e}-\kappa_{1,\e}}{{\kappa_{2,\e}}+\kappa_{1,\e}}$ follows from that for $S_\e$.   \end{proof}

\subsection{Resolving $C^{2,1}$-strictly elliptic Weyl problem via J-holomorphic curve}

We shall conclude Theorem~\ref{thm:strict-C21-Weyl} by making use of the J-holomorphic curve formulation for the generalised Weyl problem \textit{\`{a} la} Labourie~\cite{Labourie}. Recall the following equations for the mean curvature $H_\e$ or its inverse $W_\e:=H_{\e}^{-1}$ from \S\ref{subsubsec: J-hol}:
\begin{align}
dH_\varepsilon\circ J_\varepsilon^{\mathrm{II}} &= H_\varepsilon\beta_\varepsilon+2\left(H_\varepsilon^2-k_\varepsilon^2\right)\pi_\Sigma^\#\omega_\varepsilon,\label{eq:strict-Labourie-H}\\
dW_\varepsilon\circ J_\varepsilon^{\mathrm{II}} &= -W_\varepsilon\beta_\varepsilon-2\left(1-k_\varepsilon^2W_\varepsilon^2\right)\pi_\Sigma^\#\omega_\varepsilon.\label{eq:strict-Labourie-W}
\end{align}
Also, recall the uniform-in-$\e$ estimates from the same section: \begin{align}\label{eq:strict-Labourie-form-bounds}
\|\beta_\varepsilon\|_{L^\infty(\Sigma,g_\varepsilon)}&\leq C,\qquad
|\pi_\Sigma^\#d\omega_\varepsilon|_{g_\varepsilon}\leq Ck_\varepsilon W_\varepsilon \leq C'.
\end{align}
The final inequality holds by  $k_\varepsilon W_\varepsilon=k_\varepsilon/H_\varepsilon\leq1$, which follows from the arithmetic mean-geometric mean inequality. All H\"older norms below are computed in a fixed finite atlas.

Observe that the quantity on the right-hand side of \eqref{eq:strict-Labourie-H},
\begin{align*}
    H_\e^2 - k_\e^2 = \frac{\left( \kappa_{2,\e} -\kappa_{1,\e} \right)^2}{4},
\end{align*}
may play a crucial role. Thus, it is natural to argue separately for the nearly umbilical region ($\kappa_{2,\e}-\kappa_{1,\e} \approx 0$) and its complement. In the proof below, when $(\kappa_{2,\varepsilon}-\kappa_{1,\varepsilon})/(\kappa_{1,\varepsilon}+\kappa_{2,\varepsilon})$ is large, we divide \eqref{eq:strict-Labourie-W} by $1-k_\varepsilon^2W_\varepsilon^2$ and take the exterior derivative to obtain a uniformly elliptic equation for $W_\varepsilon$ in divergence form; when the ratio is small, we use \eqref{eq:strict-Labourie-H} instead and tackle the term $\pi_\Sigma^\#\omega_\varepsilon$ by estimating $\mathring{\two}_\e$ via the Codazzi equations.

\begin{proof}[Proof of Theorem~\ref{thm:strict-C21-Weyl}]
Fix any $p \in [2,\infty[$. As in Step~3 in the proof of Theorem~\ref{thm:conditional-W3p}, once we obtain the uniformly local $W^{1,p}$-bound for $H_\e$, then Theorem~\ref{thm:strict-C21-Weyl} readily follows.  More precisely, assume that for some $r_\star>0$ and $C>0$ both independent of $\e$ and $x_0\in\Sigma$, one has
\begin{equation}\label{final, to prove}
\|H_\e\|_{W^{1,p}\left(\B^{g_\e}_{r_\star}(x_0)\right)} \leq C.
\end{equation}
By a covering argument, we then have the uniform bound $\|H_\e\|_{W^{1,p}(\Sigma)} \leq C$. Hence, by bootstraping via the Gauss and Weingartan equations~\eqref{eq:coordinate-Gauss} and \eqref{eq:coordinate-Weingarten} as in the proof of Theorem~\ref{thm:conditional-W3p}, Step~3, we deduce that $f_\e$, the isometric immersions corresponding to $g_\e$, have a uniform $W^{3,p}$-bound. This completes the proof.

Now let us prove~\eqref{final, to prove}. Take a parameter $\delta_p \in ]0,1/4[$, which depends only on $p$ and is to be specified later. Then, by Proposition~\ref{prop:strict-H-and-pinching}, there exists $r_p>0$ independent of $\e$, such that 
\begin{equation}\label{a}
    \left|\frac{\kappa_{2,\varepsilon}(y)-\kappa_{1,\varepsilon}(y)}{\kappa_{1,\varepsilon}(y)+\kappa_{2,\varepsilon}(y)}-\frac{\kappa_{2,\varepsilon}(x)-\kappa_{1,\varepsilon}(x)}{\kappa_{1,\varepsilon}(x)+\kappa_{2,\varepsilon}(x)}\right|\leq\frac{\delta_p}{2}\qquad\text{for any $x,y\in\Sigma$ with $d_{g_\varepsilon}(x,y)\leq4r_p$}.
\end{equation}
After reducing $r_p$ if necessary, $\B^{g_\e}_{4r_p}(x)$ for each $x \in \Sigma$ lies in a single chart from the fixed atlas. 

In the sequel, we consider two cases separately. Again, unless otherwise specified, all the geodesic balls are taken with respect to the metric $g_\e$.

\smallskip
\noindent
{\bf Case~1: $\frac{\kappa_{2,\varepsilon}(x)-\kappa_{1,\varepsilon}(x)}{\kappa_{1,\varepsilon}(x)+\kappa_{2,\varepsilon}(x)}\leq\frac{3\delta_p}{2}$.} In this case, by \eqref{a} we have
\begin{align}\label{2delta p}
\frac{\kappa_{2,\varepsilon}(y)-\kappa_{1,\varepsilon}(y)}{\kappa_{1,\varepsilon}(y)+\kappa_{2,\varepsilon}(y)}\leq 2\delta_p\qquad\text{for all } y \in \B_{4r_p}(x).
\end{align}
By \S\ref{subsubsec: J-hol}, one has the estimates via the traceless part of the shape operator:
\begin{align}\label{eq:strict-weighted-connection}
\left|\left(H_\varepsilon^2-k_\varepsilon^2\right)\pi_\Sigma^\#\omega_\varepsilon\right| &\leq C\sqrt{H_\varepsilon^2-k_\varepsilon^2}\,\left|\nabla \left( S_\e - H_\e\, {\rm Id}\right)\right| \leq CH_\varepsilon\frac{\kappa_{2,\varepsilon}-\kappa_{1,\varepsilon}}{\kappa_{1,\varepsilon}+\kappa_{2,\varepsilon}}|\nabla S_\varepsilon|.
\end{align}

We proceed with an energy estimate. Choose a cutoff function $\chi\in C_c^\infty\left(\B_{3r_p}(x)\right)$ such that $\chi\equiv 1$ on $\B_{2r_p}(x)$ and $|\nabla\chi|\leq C/r_p$. Taking the wedge product with $\chi^p|\nabla H_\varepsilon|^{p-2}dH_\varepsilon$ on both sides of Equation~\eqref{eq:strict-Labourie-H} and integrating over $\B_{3r_p}(x)$, we deduce that 
\begin{align}\label{b}
c\int_{\Sigma}\chi^p|\nabla H_\varepsilon|^p\,\dd V_{g_\e}
&\leq C\int_{\Sigma}\chi^p|\nabla H_\varepsilon|^{p-1}\,\dd V_{g_\e}+C\int_{\Sigma}\chi^p\frac{\kappa_{2,\varepsilon}-\kappa_{1,\varepsilon}}{\kappa_{1,\varepsilon}+\kappa_{2,\varepsilon}}|\nabla S_\varepsilon||\nabla H_\varepsilon|^{p-1}\,\dd V_{g_\e}\nonumber\\
&\leq C\int_{\Sigma}\chi^p|\nabla H_\varepsilon|^{p-1}\,\dd V_{g_\e} +C\delta_p\int_{\Sigma}\chi^p|\nabla S_\varepsilon||\nabla H_\varepsilon|^{p-1}\,\dd V_{g_\e}\nonumber\\
&\leq \frac{c}{4}\int_{\Sigma}\chi^p|\nabla H_\varepsilon|^p\,\dd V_{g_\e}+C_p\delta_p^p\int_{\Sigma}\chi^p|\nabla S_\varepsilon|^p\,\dd V_{g_\e}+C_{p,r_p},
\end{align}
Thanks to the bounds in  \eqref{eq:strict-weighted-connection}, \eqref{2delta p}, and \eqref{eq:strict-Labourie-form-bounds}, Proposition~\ref{prop:strict-H-and-pinching}, as well as Young's inequality.

Now, consider as before the traceless part of the second fundamental form:
$\mathring{\mathrm{II}}_\varepsilon=\mathrm{II}_\varepsilon-H_\varepsilon g_\varepsilon$. Applying the first-order elliptic estimate for the trace-free Codazzi system~\eqref{eq:trace-free-Codazzi}, we deduce that 
\begin{align*}
\|\nabla(\chi\mathring{\mathrm{II}}_\varepsilon)\|_{L^p}
&\leq C_p\left(\|\chi\nabla H_\varepsilon\|_{L^p}+\|(\nabla\chi)\mathring{\mathrm{II}}_\varepsilon\|_{L^p}+\|\chi\mathring{\mathrm{II}}_\varepsilon\|_{L^p}+1\right)\\
&\leq C_p\|\chi\nabla H_\varepsilon\|_{L^p}+C_p\left(1+\|\nabla\chi\|_{L^\infty}\right)
\end{align*}
where $C_p$ is independent of $\chi$, $r_p$, $x$, and $\varepsilon$. This together with~\eqref{b} yields that 
\begin{align*}
\frac{3c}{4}\int_{\Sigma}\chi^p|\nabla H_\varepsilon|^p\,\dd V_{g_\e}
\leq C_p\delta_p^p\int_{\Sigma}\chi^p|\nabla H_\varepsilon|^p\,\dd V_{g_\e}+C_{p,r_p}.
\end{align*}
The constants $c$ and $C_p$ are independent of the choice of $\delta_p$ and the corresponding radius $r_p$. 
We may therefore choose $\delta_p$ so that $$C_p\delta_p^p\leq\frac{c}{4}.$$
We thus arrive at the desired $W^{1,p}$-bound for $H_\e$:
$$
\|\nabla H_\varepsilon\|_{L^p(\B_{2r_p}(x))}\leq C_{p,r_p}.
$$

\smallskip
\noindent
{\bf Case~2: $\frac{\kappa_{2,\varepsilon}(x)-\kappa_{1,\varepsilon}(x)}{\kappa_{1,\varepsilon}(x)+\kappa_{2,\varepsilon}(x)} > \frac{3\delta_p}{2}$.} In this case, 
\begin{align*}
    1-k_\varepsilon^2W_\varepsilon^2=\left(\frac{\kappa_{2,\varepsilon}-\kappa_{1,\varepsilon}}{\kappa_{1,\varepsilon}+\kappa_{2,\varepsilon}}\right)^2\geq\delta_p^2\qquad\text{on $\B_{4r_p}(x)$},
\end{align*}
so the PDE~\eqref{eq:strict-Labourie-W} for $W_\e = H^{-1}_\e$ may be recast into
\begin{align}\label{eq:strict-fixed-scale-W-divergence}
-d\left(\frac{dW_\varepsilon\circ J_\varepsilon^{\mathrm{II}}}{1-k_\varepsilon^2W_\varepsilon^2}\right) &= d\left(\frac{W_\varepsilon\beta_\varepsilon}{1-k_\varepsilon^2W_\varepsilon^2}\right)+2\pi_\Sigma^\#d\omega_\varepsilon.
\end{align}
In view of Proposition~\ref{prop:strict-H-and-pinching} and the lower bound for $1-k_\varepsilon^2W_\varepsilon^2$, this is a uniformly elliptic PDE with uniformly $C^{0,\alpha}$-bounded coefficients; the source terms are controlled by \eqref{eq:strict-Labourie-form-bounds} and $W_\varepsilon\leq c_1^{-1}$, \emph{i.e,}, the uniform $L^\infty$-bound for $H_\e$. Thus, by the $W^{1,p}$-estimates for uniformly elliptic divergence-form equations with H\"{o}lder continuous coefficients, we arrive at   
\begin{align*}
&\|\nabla W_\varepsilon\|_{L^p(\B_{2r_p}(x))}\notag\\
\leq& C_{p,\delta_p}\left\{r_p^{-1}\|W_\varepsilon\|_{L^p(\B_{3r_p}(x))}+\left\|\frac{W_\varepsilon\beta_\varepsilon}{1-k_\varepsilon^2W_\varepsilon^2}\right\|_{L^p(\B_{3r_p}(x))}+r_p\|\pi_\Sigma^\#d\omega_\varepsilon\|_{L^p(\B_{3r_p}(x))}\right\}\notag\\
\leq& C_{p,r_p,\delta_p}.
\end{align*}
See, \emph{e.g.}, Astala--Iwaniec--Martin~\cite[Sections~14.5 and~16.1]{AIM}. But $dH_\varepsilon=-H_\varepsilon^2dW_\varepsilon$ and $H_\varepsilon$ is uniformly bounded, so one also obtains the desired uniform $L^p$-bound for $\nabla H_\varepsilon$ on $\B_{2r_p}(x)$.

In both cases, we have established the bound~\eqref{final, to prove}, where $r_\star = r_p$ is independent of $\e$ and the location of  centres of the geodesic balls. The proof of Theorem~\ref{thm:strict-C21-Weyl} is therefore complete.   \end{proof}



	\bigskip
	\noindent
	{\bf Acknowledgement}. 
	The research of SL is supported by NSFC Projects 12201399, 12331008, and 12411530065, Young Elite Scientists Sponsorship Program by CAST 2023QNRC001, National Key Research $\&$ Development Programs 2023YFA1010900 and 2024YFA1014900, Shanghai Rising-Star Program 24QA2703600, Qi-Guang Scholarship, and the Shanghai Frontier Research Institute for Modern Analysis.  The research of XS is partially supported by National Key Research $\&$ Development Programs 2023YFA1010900 and 2024YFA1014900.

	\bigskip
	\noindent
	{\bf Statement of competing interests}. 
	All authors declare that there is no conflict of interest.

	\noindent
	{\bf Statement of data availability}.
	Our manuscript has no associated data.

	\noindent
	{\bf AI Statement}. The authors thank ChatGPT 5.6 AI models for theoretical derivations and computational assistance. All mathematical derivations, conclusions, and errors remain solely the responsibility of the authors.

\end{document}